\documentclass[12pt]{article}
\usepackage{amsmath,amssymb,amsthm, amsfonts}
\usepackage{etaremune, enumerate, float, verbatim}
\usepackage{algorithm}
\usepackage{algorithmic}
\usepackage{hyperref}
\usepackage{graphicx}
\usepackage{url}
\usepackage[mathlines]{lineno}
\usepackage{dsfont} 
\usepackage{tikz, graphicx, subcaption, caption} 
\usepackage[algo2e,ruled,vlined]{algorithm2e}
\usepackage{mathrsfs,amsmath}
\usepackage{color}
\definecolor{red}{rgb}{1,0,0}

\definecolor{blue}{rgb}{0,0,.7}

\definecolor{green}{rgb}{0,.6,0}

\definecolor{purp}{rgb}{.5,0,.5}

\numberwithin{figure}{section}   
\numberwithin{table}{section}   
\numberwithin{equation}{section}   

\usepackage{tikz}
\usetikzlibrary{fit}
\usetikzlibrary{backgrounds}
\usetikzlibrary{shapes.geometric}
\tikzstyle{vertex}=[circle, draw=black, thick, inner sep=0pt, minimum size=6pt]
\tikzstyle{Bvertex}=[circle, black, fill, draw, inner sep=0pt, minimum size=6pt]
\tikzstyle{vtx}=[circle, white, fill, draw=black, thick, inner sep=0pt, minimum size=6pt]
\tikzstyle{gvertex}=[circle, green, fill, draw=black, inner sep=0pt, minimum size=6pt]
\newcommand{\vertex}{\node[vertex]}

\newtheorem{thm}{Theorem}[section]
\newtheorem{cor}[thm]{Corollary}
\newtheorem{lem}[thm]{Lemma}
\newtheorem{prop}[thm]{Proposition}

\newtheorem{obs}[thm]{Observation}

\newtheorem{quest}[thm]{Question}

\theoremstyle{definition}
\newtheorem{rem}[thm]{Remark}

\theoremstyle{definition}

\theoremstyle{definition}
\newtheorem{ex}[thm]{Example}

\newcommand{\ZZ}{\mathbb{Z}}

\newcommand{\Z}{\operatorname{Z}}

\newcommand{\Zp}{\operatorname{Z}_+}

\newcommand{\pd}{\gamma_P}

\newcommand{\F}{{\mathcal F}}
\newcommand{\pt}{\operatorname{pt}}
\newcommand{\ptx}{\operatorname{pt}_Y}
\newcommand{\ptp}{\operatorname{pt}_+}

\newcommand{\ppt}{\operatorname{pt}_{\rm pd}}
\newcommand{\capt}{\operatorname{capt}}

\newcommand{\thx}{\operatorname{th}_Y}
\newcommand{\thz}{\operatorname{th}}
\newcommand{\thp}{\operatorname{th}_+}

\newcommand{\thc}{\operatorname{th}_c}
\newcommand{\thpd}{\operatorname{th}_{\rm pd}}
\newcommand{\thxx}{\operatorname{th}_Y^\times}
\newcommand{\thxa}{\operatorname{th}_Y^\ast}
\newcommand{\thzx}{\operatorname{th}^\times}
\newcommand{\thza}{\operatorname{th}^\ast}
\newcommand{\thpx}{\operatorname{th}_+^\times}
\newcommand{\thpa}{\operatorname{th}_+^\ast}
\newcommand{\thcx}{\operatorname{th}_c^{\x}}
\newcommand{\thca}{\operatorname{th}_c^\ast}
\newcommand{\thpda}{\operatorname{th}_{\rm pd}^\ast}
\newcommand{\thpdx}{\operatorname{th}_{\rm pd}^\times}
\newcommand{\ecc}{\operatorname{ecc}}

\newcommand{\rd}{\operatorname{rd}} 
\newcommand{\kx}{k_Y}
\newcommand{\kp}{k_+}

\newcommand{\dist}{\operatorname{dist}}

\newcommand{\rad}{\operatorname{rad}}

\newcommand{\HH}{\mathcal{H}}

\newcommand{\n}{\{1,\dots,n \}}
\newcommand{\x}{\times}

\newcommand{\bit}{\begin{itemize}}
\newcommand{\eit}{\end{itemize}}
\newcommand{\ben}{\begin{enumerate}}
\newcommand{\een}{\end{enumerate}}
\newcommand{\beq}{\begin{equation}}
\newcommand{\eeq}{\end{equation}}
\newcommand{\bea}{\begin{eqnarray*}} 
\newcommand{\eea}{\end{eqnarray*}}
\newcommand{\bpf}{\begin{proof}}
\newcommand{\epf}{\end{proof}\ms}
\newcommand{\bmt}{\begin{bmatrix}}
\newcommand{\emt}{\end{bmatrix}}
\newcommand{\ms}{\medskip}

\newcommand{\cp}{\, \Box\,}
\newcommand{\lc}{\left\lceil}
\newcommand{\rc}{\right\rceil}
\newcommand{\lf}{\left\lfloor}
\newcommand{\rf}{\right\rfloor}
\newcommand{\lp}{\!\left(}
\newcommand{\rp}{\right)}
\newcommand{\lb}{\left[}
\newcommand{\rb}{\right]}

\newcommand{\ds}{\displaystyle}
\newcommand{\du}{\,\dot{\cup}\,}

\title{Product Throttling\footnote{This is a preprint of the following chapter: Sarah E. Anderson, Karen L. Collins, Daniela Ferrero, Leslie Hogben, Carolyn Mayer, Ann N. Trenk, and Shanise Walker, Product Throttling, published in Research Trends in Graph Theory and Applications, edited by D. Ferrero, L. Hogben, S. Kingan, and G. Mathews, 2021, Springer reproduced with permission of Springer Nature Switzerland AG. The final authenticated version is available online at: \url{https://doi.org/10.1007/978-3-030-77983-2}.}}
\author{Sarah E. Anderson\thanks{Department of Mathematics, University of St. Thomas, St. Paul, MN 55105, USA (ande1298@stthomas.edu).} \and Karen L. Collins\thanks{Department of Mathematics and Computer Science, Wesleyan University, Middletown, CT 06459, USA (kcollins@wesleyan.edu).}\and Daniela Ferrero\thanks{Department of Mathematics, Texas State University, San Marcos, TX 78666, USA (dferrero@txstate.edu).}\and Leslie Hogben\thanks{Department of Mathematics, Iowa State University, Ames, IA 50011, USA and American Institute of Mathematics,  
San Jose, CA 95112, USA (hogben@aimath.org).}\and Carolyn Mayer \thanks{Department of Mathematical Sciences, Worcester Polytechnic Institute, Worcester, MA 01609, USA and Sandia National Laboratories, Albuquerque, NM 87185, USA (cdmayer@sandia.gov).}\and Ann N. Trenk \thanks{Department of Mathematics, Wellesley College, Wellesley, MA 02481, USA (atrenk@wellesley.edu).}\and Shanise Walker\thanks{Department of Mathematics, University of Wisconsin-Eau Claire, Eau Claire, WI 54701, USA (walkersg@uwec.edu).}}

\begin{document}
\maketitle

\section{Introduction} \label{s:intro}

Throttling addresses the question of minimizing the sum or the product of the resources used to accomplish a task  and the time needed to complete that task for various graph searching processes.  Graph parameters of interest include various types of zero forcing, power domination, and Cops and Robbers.  

The resources used to accomplish a task can be blue vertices 
in zero forcing,  Phasor Measurement Units (PMUs) in power domination, or cops in Cops and Robbers.
The time is the number of rounds needed to complete the process (the propagation time or capture time).  

We begin by defining the graph parameters for which we will discuss product throttling. Our focus is on  connected graphs of order at least two (unless otherwise stated). 
Zero forcing is a coloring game on a graph, where the goal is to color all the vertices blue (starting with  each vertex colored blue or white). White vertices are then colored blue by applying a {color change rule}; the type of zero forcing is determined by the color change rule.  Standard zero forcing uses the \emph{standard color change rule}:
\bit
\item[] If $w$ is the unique white neighbor of a blue vertex $v$, then  change the color of  $w$ to blue.  
\eit
Positive semidefinite (PSD) zero forcing uses the \emph{PSD color change rule}:
\bit
\item[] Let $B$ be the set of (currently) blue vertices and let $W_1,\dots, W_k$ be the sets of vertices of the  components of $G-B$.  If $v\in B$, $w\in W_i$, and $w$ is the only white neighbor of  $v$ in $G[W_i\cup B]$, then change the color of $w$ to blue. 
\eit
Note that it is possible that there is only one component of $G-B$, and in that case the effect of the PSD color change rule is the same as that of the standard color change rule.   

A nonempty set $S\subseteq V(G)$ defines an initial set of  blue vertices (with all vertices not in $S$ colored white); this is called an \emph{initial coloring of $G$}. Given an initial  coloring $S$ of  $G$, the   \emph{final coloring}  of $S$ is the set of blue vertices obtained by  applying the color change rule until no more changes are possible (other names for the final coloring include the \emph{derived set}  and the \emph{closure} of $S$).
A set $S$ is a \emph{standard zero forcing set} (respectively, \emph{PSD zero forcing set}) of $G$ if the final coloring of $S$ is $V(G)$  using the standard (respectively, PSD color change rule).    The \emph{standard zero forcing number}  (respectively, \emph{PSD zero forcing number}), denoted by $\Z(G)$ (respectively, $\Zp(G)$) is the minimum cardinality of a standard zero forcing set (respectively, a PSD zero forcing set).   
  Hereafter, we will use the term \emph{forcing set} to mean standard or PSD zero forcing set. 
  
If $v$ is used to change the color of $w$ by a color change rule, we say \emph{$v$ forces $w$} and write $v\to w$.
For a given forcing set $S$, we construct the  final coloring, recording the forces.  Depending on context, the symbol $\F$ is used to denote the \emph{set  of forces} that produces the final coloring, or an ordered list of forces (in the order they were performed), called a \emph{chronological list of forces}.  
For a given  set $S$, there are often choices as to which vertex forces a particular vertex, so a set of  forces, or a chronological list of forces, is usually not unique.  However, the final coloring is unique for standard and  PSD  zero forcing \cite{AIM, HLS21+}. 

We can also approach (standard or PSD) zero forcing not as an individual sequence of forces but via \emph{rounds}, where in each round we perform  all possible forces that can be done independently of each other (rounds are also called \emph{time steps} in the literature).  
Starting with $S\subseteq V(G)$, we define two sequences of sets, the set $S^{(i)}$ of vertices that turn blue in round $i$   and  the set $S^{[i]}$ of vertices that are blue after round $i$. Thus $S^{[0]}=S^{(0)}=S$ is the initial set of blue vertices.  Assume $S^{(i)}$ and $S^{[i]}$ have been constructed and $S^{(i)}\ne \emptyset$. Then
 \[S^{(i+1)}=\{w:\mbox{$w$ can be forced by some $v$ (given  $S^{[i]}$  blue)}\}\ \ \mbox{  and }\ \ S^{[i+1]}=S^{[i]}\cup S^{(i+1)}.\]  
 Let $p$ denote the greatest integer such that  $S^{(p)}\ne \emptyset$. Since $S^{(i)}= \emptyset$ implies $S^{(i+1)}= \emptyset$, $S$ is a  forcing set of $G$ if and only if  $S^{[p]}=V(G)$.  When $S$ is a  forcing set, this  $p$ is called the \emph{propagation time} of $S$ in $G$, denoted $\pt(G;S)$ or $\ptp(G;S)$ for standard and PSD zero forcing set, respectively; if $S$ is not a forcing set, then $\pt(G;S)=\infty$ or $\ptp(G;S)=\infty$.   For $k\in\ZZ^+$, $\pt(G,k)=\min_{|S|=k}\pt(G;S)$ and $\ptp(G,k)=\min_{|S|=k}\ptp(G;S)$.  
 The \emph{standard propagation time of $G$} (respectively, \emph{PSD propagation time of $G$}) is $\pt(G)=\pt(G,\Z(G))$ (respectively, $\ptp(G)=\ptp(G,\Zp(G))$). 

 For each $v \in V(G)$,  define the {\em round  function} by     
 $\rd(v)=k$ for $v\in S^{(k)}$. 
 A  \emph{propagating} set of forces is one in which $\rd(u)<\rd(v)$ implies $u$ is forced before $v$ in the associated chronological list of forces, and this is the only kind of forcing set we are concerned with.
The round function will also be used for power domination and Cops and Robbers, but the meaning will be clear from the context or a subscript will be added to identify the parameter.

The name \emph{zero forcing} comes from the fact that the process describes forcing zeros in the null vector of a symmetric  matrix using only the pattern of off-diagonal nonzero entries of the matrix (a graph describes the nonzero  off-diagonal pattern of a symmetric matrix).  The zero forcing number was introduced in  \cite{AIM} as an upper bound for the maximum nullity, or equivalently, maximum multiplicity of an eigenvalue, among real symmetric matrices having this graph.  Zero forcing
was introduced independently in mathematical physics in the study of control of quantum systems \cite{graphinfect}, and later reintroduced as fast mixed graph searching \cite{Y13}.  Arguably its first appearance was as part of the power domination process, which we describe next.

Power domination models the observations that can be made by PMUs and was studied  using graphs by Haynes et al.~in
\cite{HHHH02}; Brueni and Heath \cite{BH} showed that a simplified version of the propagation rules  is equivalent to the original version in \cite{HHHH02}, and we use their propagation rules. For a nonempty set $S$ of vertices of $G$, $N[S]$ denotes the closed neighborhood of $S$.  A set $S$ is a \emph{dominating set} of $G$ if $N[S]=V(G)$, and the minimum cardinality of a dominating set is the \emph{domination  number} of $G$, denoted by $\gamma(G)$.
Given $S\subseteq V(G)$, define the sequences of sets $P^{(i)}(S)$ and $P^{[i]}(S)$ by the following recursive rules: 
\ben[(1)] 
\item\label{domstep} $P^{[0]}(S)=P^{(0)}(S)=S$, $P^{[1]}(S)=N[S]$ and $P^{(1)}(S)=N[S]\setminus S$.  \vspace{-3pt}
\item\label{zfstep} For $i\ge 1$,\vspace{-3pt}
\bea P^{(i+1)}(S)&=& \big \{w\in V(G)\setminus P^{[i]}(S) : \exists u\in P^{[i]}(S), \,N_G(u)\setminus P^{[i]}(S)=\{w\} \big \},\\ \vspace{-3pt}
P^{[i+1]}(S)&=&P^{[i]}(S)\cup P^{(i+1)}(S).\vspace{-3pt}
\eea
\een\vspace{-6pt}

Step \eqref{domstep} is called the domination step, because it results in $P^{[1]}(S)=N[S]$. Step \eqref{zfstep} is called the zero forcing  step, because  $P^{(i+1)}(S)=N[S]^{(i)}$ for $i\ge 1$.   
For $v\in  P^{(i)}(S)$, we say \emph{$v$ is observed in round $i$} or $\rd(v)=i$ (if necessary to distinguish from zero forcing, we write $\rd_{pd}(v)$). 
If every vertex is observed in some round, i.e., there is an $i$ such that $P^{[i]}(S)=V(G)$,  then $S$ is a \emph{power dominating set} of $G$; $S$ is a power dominating set of $G$ if and only if $N[S]$ is a zero forcing set of $G$. The \emph{power domination number} of   $G$, denoted by $\pd(G)$, is the minimum cardinality of a power dominating set. 
 When $S$ is a power dominating set, the least positive integer $p$ with the property that $P^{[p]}(S)=V(G)$ is the {\it power propagation time} of $S$ in $G$, denoted by $\ppt(G;S)$; if $S$ is not a power dominating set, then $\ppt(G;S)=\infty$. Observe that $\ppt(G;S)=\pt(G;N[S])+1$.
  For $k\in\ZZ^+$, $\ppt(G,k)=\min_{|S|=k}\ppt(G;S)$ and the \emph{power propagation time of $G$} is $\ppt(G)=\ppt(G,\pd(G))$.

 Cops and Robbers is a two-player game played on a graph.  One player places and moves a collection of cops and the other  places and moves a single robber. The goal for the cops is to capture the robber by having a cop occupy the same vertex the robber occupies.  The goal of the robber is to avoid capture.  After an initial placement of the cops on a multiset of vertices (meaning more than one cop can occupy a single vertex), followed by the placement of the robber, the game is played in a sequence of  rounds during which the players take turns, both playing in a single round: The team of cops takes a turn by allowing each cop to move to an adjacent vertex or stay in place. Similarly, the robber takes a turn by moving to an adjacent vertex or staying in place. The cops win the game if after some finite number of rounds, a cop captures the robber.  If the robber has a strategy to evade the cops indefinitely, the robber wins.   The \emph{cop number} $c(G)$ of a graph $G$ is  the minimum number of cops required to capture the robber playing on $G$ \cite{AF84}.  The \emph{capture time}, denoted $\capt(G)$,  is the number of rounds it takes  for $c(G)$ cops to capture the robber on the graph $G$ (assuming all players follow optimal strategies) \cite{BGHK09}, and for any $k \geq c(G)$, the \emph{$k$-capture time} of $G$, denoted by $\capt_k(G),$ is the minimum number of rounds it takes for $k$ cops to capture the robber on $G$ (assuming that all players follow optimal strategies) \cite{BPPR17}.  If $k<c(G)$, then $\capt_k(G)=\infty$. 
 
Throttling originated with a question of Richard Brualdi to Michael Young in a talk about zero forcing and propagation time at the 2011 International Linear Algebra Society Conference in Braunschweig, Germany.
This led Butler and Young to initiate the study of sum throttling  for (standard) zero forcing in \cite{BY13}. 
Sum throttling has been studied for numerous parameters including standard zero forcing, PSD zero forcing, and their minor monotone floors; power domination; Cops and Robbers (see \cite[Chapter 10]{HLS21+} for a survey).  Here we define sum throttling for the four graph games we discuss, i.e., standard zero forcing, PSD zero forcing, power domination, and Cops and Robbers.
 The  \emph{standard throttling number} 
 of $S$ in $G$ is $\thz(G;S)=|S|+\pt(G;S)$, and the \emph{standard $k$-throttling number} is $\thz(G,k)= k +  \pt(G,k)$. 
The    \emph{standard throttling number} of  $G$ is  
\[\thz(G)=\min_{S\subseteq V(G)} \thz(G;S)=\min_{\Z(G)\le k\le n}\thz(G,k). \]  
 The  \emph{PSD throttling number}  of $S$ in $G$ is $\thp(G;S)=|S|+\ptp(G;S)$, and the \emph{PSD $k$-throttling number} is $\thp(G,k)= k +  \ptp(G,k)$. 
The    \emph{PSD throttling number} of  $G$ is  
\[\thp(G)=\min_{S\subseteq V(G)} \thp(G;S)=\min_{\Zp(G)\le k\le n}\thp(G,k). \]  
The  \emph{power domination throttling number}  of $S$ in $G$ is $\thpd(G;S)=|S|+\ppt(G;S)$, and the \emph{power domination $k$-throttling number} is $\thpd(G,k)= k +  \ppt(G,k)$. 
The    \emph{power domination throttling number} of  $G$ is  
\[\thpd(G)=\min_{S\subseteq V(G)} \thpd(G;S)=\min_{\pd(G)\le k\le n}\thpd(G,k). \]  
The  \emph{cop throttling number}  of $S$ in $G$ is $\thc(G;S)=|S|+\capt(G;S)$, and the \emph{cop $k$-throttling number} is $\thc(G,k)= k +  \capt_k(G)$. 
The    \emph{cop throttling number} of  $G$ is  
\[\thc(G)=\min_{S\subseteq V(G)} \thc(G;S)=\min_{c(G)\le k\le n}\thc(G,k). \]  

Product throttling minimizes a product of the number of vertices and the propagation time. In order to make product throttling interesting, the case of a zero product (where each vertex has a cop/PMU/blue color and the propagation time is zero) must be excluded.  This can be done by requiring that the cost of positioning cops/PMUs/blue vertices  be considered (e.g., by adding one to the number of rounds before multiplying by the number of vertices used), and  we describe this as {\em product throttling with initial cost}. 
Alternatively, a requirement that at least one round be performed must be added  or the initial set consisting of all vertices must be excluded.  
The study of product throttling  was initiated in  \cite{cop-throttle2} for Cops and Robbers. It was assumed there is a time cost to placing cops and the product throttling number was defined as the the number of cops times one more than the propagation time.  
 In contrast, PMUs remain in place but it is natural to assume that the domination step always occurs. In the study of product throttling for power domination in \cite{product-power-throt}, at least one round was required, and the product throttling number was defined as the product of the number of PMUs and the power propagation time. 
Formal versions of these definitions are given in Section \ref{s:prod-CR} for Cops and Robbers and Section \ref{s:prod-power-a} for power domination. 

Requiring at least one round and excluding the zero round case are effectively the same for connected graphs of order at least two, but it is more convenient to exclude the zero round case by requiring that the number of vertices used is less than the order; 
 this avoids having two different definitions of propagation time. We refer to this as {\em product throttling with no initial cost}.  This is the approach taken   in Section \ref{s:prod-univ}, where   formal versions of the two definitions in  universal notation  are presented  and discussed further, and in subsequent sections.   
Product throttling for the standard zero forcing number  is not interesting for the first definition and not really a throttling question for the second definition (see Section \ref{s:prod-Z}). This fact (or at least the need to exclude the zero solution and that first definition  results in product throttling number $n$ for every graph of order $n$) may have delayed the introduction of product throttling. Using the universal perspective, we examine both definitions of product throttling for Cops and Robbers in Section \ref{s:prod-CR-a}, for power domination in Section \ref{s:prod-power-2}, and  for PSD zero forcing in   Section \ref{s:prod-Z-PSD}.   Section \ref{s:compare} compares product throttling for Cops and Robbers, power domination, and PSD zero forcing.

We need some additional notation.  The path, cycle, and complete graph on $n$ vertices are denoted by $P_n$,  $C_n$, and  $K_n$, respectively, and $K_{r,n-r}$ denotes a complete bipartite graph. For a graph $G$, $\alpha(G)$ denotes the independence number of $G$.
For all the graph parameters discussed, the number of rounds is at least the maximum distance of any vertex to the initial set $S$, and this plays an important role in the analysis of throttling. Let $S$ be a set of vertices of $G$.  For $v\in V(G)$, the \emph{distance} from $v$ to $S$ is $\dist(S,v)=\min_{x\in S}\dist(x,v)$. The {\em eccentricity}    of $S$ is defined by $\ecc(S)=\max_{v\not\in S}\dist(S,v).$ The {$k$-radius} of $G$ is $\rad_k(G)=\min_{S\subseteq V(G), |S|=k}\ecc(S).$  


\section{Initial cost product throttling for Cops and Robbers}
\label{s:prod-CR}
In this section,  we follow the convention in \cite{cop-throttle2} and do not assume $G$ is connected or that its  order $n$ is  at least two unless stated otherwise.  The \emph{product cop throttling number with initial cost} defined in   \cite{cop-throttle2} is
\[\ds \thcx(G)=  \min_{c(G)\le k\le n}\{ k (1+ \capt_k(G))\} \]
or equivalently,  $\thcx(G)=\min_{S\subseteq V(G)}\thcx(G;S)$ 
where $\thcx(G;S)=  |S| (1+ \capt(G;S))$.
This choice of definition for Cops and Robbers reflects the fact that there is a time cost to getting the cops in position.  In this section we summarize results from {\rm \cite{cop-throttle2}}.

\begin{obs} \label{o:thcx-bds}{\rm \cite{cop-throttle2}}
There are two immediate upper bounds:
\ben[$(1)$]
\item $\thcx(G)\le c(G)(1+\capt(G))$.
\item $\thcx(G)\le 2\gamma(G)$.
\een
\end{obs}

\begin{rem}\label{r:cr-sum=prod}\cite{cop-throttle2}
Let $G$ be a graph of order $n$  
and suppose $S\subseteq V(G)$ with $|S|\ge c(G)$.  Then  
\[\thcx(G;S)=|S|(1+ \capt(G;S))=|S|+|S|\capt(G;S)\ge |S|+\capt(G;S)=\thc(G;S),\] so $  \thcx(G)\ge \thc(G)$.  Equality occurs exactly when $|S|=1$ or $\capt(G;S)=0$, so $\thcx(G)=\thc(G)$ if and only if $\thc(G)=\thc(G;S)$ with $|S|=1$ or $|S|=n$, i.e., when the cop throttling number can be realized with a single cop or a cop
on every vertex. 
\end{rem}

The product cop throttling number can be determined for low values by factoring the proposed value.

\begin{prop}\label{p:thcx-low}{\rm\cite{cop-throttle2}}  Let $G$ be a graph  of order $n$.
\ben[$(1)$]
	\item\label{thcx1} $\thcx(G)=1$ if and only if  $ \thc(G)=1$ if and only if $G=K_1$.  
	\item\label{thcx2} $\thcx(G)=2$ if and only if $ \thc(G)=2$ if and only if either $G=2K_1$ or $\gamma(G)=1$.  
	\item\label{thcx3} $\thcx(G)=3$ if and only if   
	 	$G$ satisfies at least one of the following conditions:
		\begin{enumerate}[(a)]
			\item \label{30} $G=3K_1$ or $G=K_1\du K_2$.
			\item \label{12} $\gamma(G)\ge 2$ and there exists $z\in V(G)$ such that
 				\ben[(i)]
 					\item \label{12a} for all $v\in V(G)$, $\dist(z,v)\le 2$, and 
					 \item \label{12b} for all $w\in V(G)\setminus N[z]$, there is a vertex $u\in N[z]$ such that $N[w]\subset N[u]$.   
				\een
		\end{enumerate}

\item\label{thcx4} $\thcx(G)=4$ if and only if $G$ satisfies at least one of the following conditions:
		\ben[(a)]
			\item $G=4K_1$ or $G=2K_1\du K_2$. 
			\item $\gamma(G)= 2$ and $\capt_1(G)\ge 3$.
			\item $c(G)=1$ and $\capt(G)=3$.
		\een
\een
\end{prop}

It is immediate that $\thcx(K_n)=2$ and $\thcx(K_{1,n-1})=2$ for $n\ge 2$.

A graph $G$ is a {\em chordal graph} if it has no induced cycle of length greater than 3. The next result is less elementary than the previous ones.
\begin{thm}\label{t:thcx-chord} {\rm \cite{cop-throttle2}} Let $H$ be a chordal graph.  Then $\capt_k(H)=\rad_k(H)$.
Furthermore, \[\thcx(H)=1+\rad(H)=c(H)+\capt(H).\] 
\end{thm}

From Theorem \ref{t:thcx-chord}, $\thcx(P_n)=1+\lc\frac{n-1}2\rc$, and more generally, $\thcx(T)=1+\rad(T)$ for any  tree $T$.

Theorem \ref{t:thcx-chord} provides many examples of graphs $G$ with $\thcx(G)= c(G)(1+\capt(G))$, thus achieving equality in the first upper bound in Observation \ref{o:thcx-bds}.  
It can also be the case that  $\thcx(G)$ is realized by small capture time and  a larger number of cops, e.g.,  by using $\gamma(G)$ cops.  One example of this is provided by a graph in the family $H(n)$ defined in \cite{BGHK09}; it is shown there that $c(H(n))=1$ and $\capt(H(n))=n-4$.  It as observed in  \cite{cop-throttle2} that for $H(11)$ (see Figure \ref{f:H11}), $\capt(H(11))=7$, but vertices 5 and 7 dominate the graph, so  $\thcx(H(11))=2(1+1)=4$.  

\begin{figure}[h]
\centering
\begin{tikzpicture}[scale=.7]
\vertex (v1) at (1,0)[label=below: \text{\footnotesize$1$}] {};
\vertex (v2) at (-1,0) [label=below:\text{\footnotesize$2$}]{};
\vertex (v3) at (0,-1) [label=right:\text{\footnotesize$3$}]{};
\vertex (v4) at (2,0)[label=below:\text{\footnotesize$4$}]{};
\vertex (v5) at (0,1)[label=right:\text{\footnotesize$5$}]{};
\vertex (v6) at (-2,0)[label=below:\text{\footnotesize$6$}]{};
\vertex (v7) at (0,-2)[label=right:\text{\footnotesize$7$}]{};
\vertex (v8) at (3,0)[label=below:\text{\footnotesize$8$}] {};
\vertex (v9) at (0,2)[label=above:\text{\footnotesize$9$}]{};
\vertex (v10) at (-3,0)[label=below:\text{\footnotesize$10$}]{};
\vertex (v11) at (0,-3)[label=below:\text{\footnotesize$11$}]{};
\foreach \n in {2,3,4,5}{ \draw[thick] (v1) to (v\n); }
\foreach \n in {3,5,6}{ \draw[thick] (v2) to (v\n); }
\foreach \n in {3,5,7,8}{ \draw[thick] (v4) to (v\n); }
\draw[thick] (v5) to (v9);
\draw[thick] (v7) to (v3);
\draw[thick] (v7) to (v11);
\foreach \n in {9,5,3,7,10}{ \draw[thick] (v6) to (v\n); }
\foreach \n in {5,7,9,11}{ \draw[thick] (v8) to (v\n); }
\foreach \n in {7,9,11}{ \draw[thick] (v10) to (v\n); }
\end{tikzpicture}
\caption{The graph $H(11)$.  \label{f:H11}}
\end{figure}
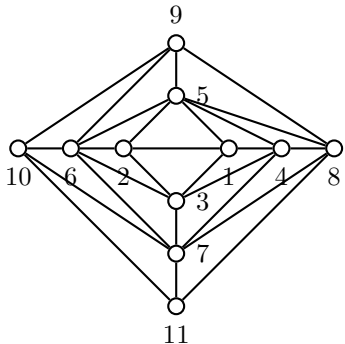

However, the next example provides a family of graphs $G$ for which both $\thcx(G,c(G))>\thcx(G)$ and $\thcx(G,\gamma(G))>\thcx(G)$ for sufficiently large order.
Fix a positive integer $r$, define $M'(r)$ to be the graph that is the union of $C_4$ and three disjoint copies of $P_{r+1}$ where one of the end points of each of the paths is on a distinct vertex of
$C_4$,  and define  $M(r)=M'(r)\circ K_1$.   The graph $M(3)$ is shown in Figure \ref{fig:L3}; the order of $M(r)$ is $6r+8$.

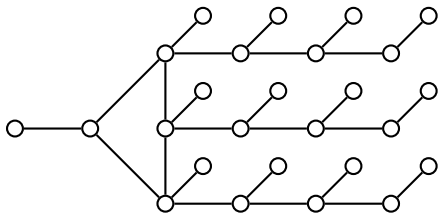
\begin{figure}[h]
\centering
\begin{tikzpicture}[scale=.5]

 \vertex (A) at (-2,4) {};
 \vertex (B) at (0,4) {};
\foreach \row in {1,2,3}{
     \foreach \col in {1,2,3,4}{
           \vertex (c\col_r\row) at (2*\col,2*\row) {};
           \vertex (c\col_r\row_2) at (2*\col+1,2*\row+1) {};
           \draw[thick] (c\col_r\row) to (c\col_r\row_2);
          }
     \draw[thick] (c1_r\row) to (c2_r\row);
     \draw[thick] (c2_r\row) to (c3_r\row);
     \draw[thick] (c3_r\row) to (c4_r\row);
      }
 \draw[thick] (c1_r1) to (c1_r2);
 \draw[thick] (c1_r2) to (c1_r3);
 \draw[thick] (A) to (B);
 \draw[thick] (B) to (c1_r1);
 \draw[thick] (B) to (c1_r3);
\end{tikzpicture}\caption{The graph $M(3)$.  \vspace{-8pt} \label{fig:L3}}
\end{figure}

\begin{prop}{\rm \cite{cop-throttle2}} .    
For $r\ge 7$, $\thcx(M(r))<c(M(r))(1+\capt(M(r)))$ and $\thcx(M(r))<2\gamma(M(r))$.
\end{prop}


\section{\large Product throttling for power domination with no initial cost}\label{s:prod-power-a}

Let $S$ be a power dominating set. In the papers that studied power propagation time and in this chapter,   the {power propagation time} of $S$ in $G$ is defined to be the least nonnegative integer $p$ such that $P^{[p]}(S)=V(G)$ and is denoted by $\ppt(G;S)$. Thus $\ppt(G;V(G))= 0$.  When product throttling for power domination was introduced in \cite{product-power-throt}, the perspective was that the domination step always takes place, so the power propagation time of $S$ is at least one even if $S=V(G)$.  That is, the definition of power propagation time was modified to require $p$  to be positive. 
Observe that $\ppt(G;S)\ge 1$  for all $S\ne V(G)$.  For power domination (and all other parameters discussed), it is immediate that $\ppt(G,n-1)=1$ when $G$ is a connected graph and has order $n\ge 2$.  Thus the restriction that $\ppt(G,k)$ be positive can be achieved by  allowing the value zero for  $\ppt(G,k)$, excluding $k=|V(G)|$ from the definition of product power throttling, and requiring that $G$ be connected  and have order at least two.

For a connected graph $G$ of order $n\ge 2$, the \emph{product power throttling number with no initial cost}\footnote{Here `no initial cost' refers to the act of monitoring, not the obvious initial cost of PMU placement.} is
$ \thpda(G)=  \min_{1\le k < n} k  \ppt(G,k); $
equivalently, $\thpda(G)=\min_{S\subsetneq V(G)}\thpda(G;S)$ 
where $\thpda(G;S)=  |S|  \ppt(G;S)$ for $S\ne V(G)$.  Note that most results in \cite{product-power-throt} are for connected graphs but do not assume order at least two, but here we assume the order is at least two. Note that what we here denote by $\thpda(G)$ is denoted by $\thpdx(G)$ in \cite{product-power-throt}.
It is immediate that
\[\ds \thpda(G)=  \min_{\pd(G)\le k \le\gamma(G)} k  \ppt(G,k). \]

In this section, we summarize results from~\cite{product-power-throt}. 
The next observation gives lower and upper bounds for $\thpda(G)$. 
\begin{obs}\label{o:basic-bds-up}\label{o:basic-bds} {\rm \cite{product-power-throt}} Let $G$ be a connected graph of order at least two. Then  
\ben[$(1)$]
\item $\thpda(G)\ge \pd(G)\ge 1$. 
\item $\thpda(G)\leq \gamma(G)$.
\item $\thpda(G)\leq \pd(G)\ppt(G)$.
\een
\end{obs}
The next result uses the maximum degree of a graph, denoted by $\Delta (G)$,  to give a lower bound for $\thpda(G)$. 
\begin{prop}{\rm \cite{product-power-throt}}\label{t:lowerbd} 
Let $G$ be a connected graph of order at least two. Then \[ \thpda(G) \ge \lc \frac {|V(G)|} {\Delta(G)+1}\rc.\]
\end{prop}
Proposition~\ref{t:lowerbd} follows as a corollary of Theorem 2.1 in {\rm \cite{FHKY17}. 

\subsection{Graphs with $\thpda(G)=\gamma(G)$}
The following results from~\cite{product-power-throt} summarize conditions sufficient to ensure  $\thpda(G)=\gamma(G)$ and give families of graphs for which  $\thpda(G)=\gamma(G)$.

\begin{obs}{\rm \cite{product-power-throt}}\label{c:squeeze} In any connected graph $G$ of order $n\ge 2$ with $\gamma(G)=\lc\frac n {\Delta(G)+1}\rc$, \[\thpda(G)=\lc\frac n {\Delta(G)+1}\rc.\]
\end{obs}

Observation~\ref{c:squeeze} follows as a result of Proposition~\ref{t:lowerbd}. The next result gives two families of graphs for which the equality in Observation~\ref{c:squeeze} holds.
 
\begin{obs}\label{p:pathcycle} {\rm \cite{product-power-throt}} Let $n\ge 2$.  Then $\thpda(P_n)=\lc \frac n 3\rc$ because $\gamma(P_n)= \lc \frac n 3\rc$ and $\Delta(P_n)=2$.  Similarly,  $\thpda(C_n)=\lc \frac n 3\rc$ since $\gamma(C_n)= \lc \frac n 3\rc$ and $\Delta(C_n)=2$, $\thpda(C_n)=\lc \frac n 3\rc$.
\end{obs}

\begin{prop}\label{basic} {\rm \cite{product-power-throt}} Let $G$ be a connected graph of order at least two. Then $\pd(G)=\gamma(G)$ if and only if $\ppt(G)=1$. In this case  $\thpda(G)=\gamma(G)$.\end{prop}

\begin{cor} Let $n\ge 2$. 
\ben[$(1)$]
\item $\thpda(K_n)=1=\gamma(K_n)$.
\item $\thpda(K_{1,n-1})=1=\gamma(K_{1,n-1})$ and  $\thpda(K_{r,n-r})=2=\gamma(K_{r,n-r})$ for $2\le r\le n-2$.
\een
\end{cor}

A graph $G$ is a \emph{unit interval graph} if there is a representation that assigns a closed unit length real interval $I(v)$ to each $v \in V(G)$ such that  $I(x) \cap I(y) \neq \emptyset$   if and only if $xy \in E(G)$ for $x, y \in V(G)$.  
It was shown in \cite{product-power-throt} that the  product power throttling number of a unit interval graph is its domination number.  
 
   \begin{thm}
 \label{t:unit-int-pd} {\rm \cite{product-power-throt}} 
 If $G$ is a connected unit interval graph, then $\thpda(G) = \gamma(G)$.
 \end{thm}

\subsection{Cartesian products}

The \emph{Cartesian product} $G\Box H$ of graphs $G$ and $H$ is the graph whose vertex set is $V(G\Box H) = V(G) \times V(H)$ where two vertices $(x_1,y_1)$ and $(x_2,y_2)$ are adjacent in $G \Box H$ if either $x_1=x_2$ and $y_1y_2\in E(H)$ or $y_1=y_2$ and $x_1x_2\in E(G)$. Bounds on the product power throttling number of a Cartesian product were presented  in \cite{product-power-throt} and following results were shown.

\begin{prop}{\rm \cite{product-power-throt}}  
For $1\le n\le m$, $ \thpda(K_n \Box K_m) =\gamma(K_n \Box K_m)=n$. \end{prop}

\begin{prop}\label{completeprod-p1}{\rm \cite{product-power-throt}} 
Let $H$ be a connected graph of order $n$ and let $G = H \Box K_m$ with $m \ge \Delta(H)(n - 1)+1$.  Then $\thpda(G) =n=\gamma(G)$.
In particular, if  $H = C_n$ or $P_n$ and $m \ge 2n - 1$, then $ \thpda(G) =n$. 
\end{prop}

\begin{thm}\label{t:grid}{\rm \cite{product-power-throt}}  For all $n\geq 1$ and $m\ge 2$,   $\thpda(J_n\square J_m)=\gamma(J_n\square J_m)$, where $J_r=P_r$ or $J_r=C_r$ for $r\ge 3$ and $J_r=P_r$ for $r=1,2$.
\end{thm}

\subsection{Extreme values}

\begin{prop}\label{thpda:low}{\rm \cite{product-power-throt}}   Let $G$ be a connected graph of order at least two.
\ben[$(1)$]
\item \label{thpda1} $\thpda(G)=1$ if and only if $\gamma(G)=1$. 
\item  \label{thpda2}  $\thpda(G)=2$ if and only if $G$ satisfies at least one of the following conditions:
\ben[(a)] 
\item \label{i:d=2} $\gamma(G)=2$.  
\item \label{i:pt=2} $\pd(G)=1$ and $\ppt(G)=2$.  
\een
\een
\end{prop}

A construction for creating any connected graph $G$ with $\pd(G)=1$ and $\ppt(G)=2$  appears  in \cite{product-power-throt}. 

\begin{thm}{\rm \cite[Theorem 2.1]{dom-book}}\label{half}
Let $G$ be a graph of order $n$ having no isolated vertices. Then $\gamma (G) \leq \lf\frac n 2\rf$.
\end{thm}

Theorem~\ref{half} and Observation~\ref{o:basic-bds} provide an upper bound that  for any connected graph $G$ of order $n\ge 2$: $\thpda(G)\le \gamma(G)\le \frac{n}{2}$.  

For any graph $H$, define the {\em corona} of $H$ with $K_1$,  denoted by $H\circ K_1$, to be the graph obtained from $H$ by appending a leaf to each vertex of $H$. The next result will be used in Section \ref{sfamilies} and was used in \cite{product-power-throt} to characterize graphs having  $\thpda(G)=\frac{n}{2}$ (see Theorem \ref{half-ratio-iff}).

\begin{thm}\label{half-ratio}{\rm \cite{product-power-throt}} 
If $H$ is a connected graph of order at least two and $G=H\circ K_1$, then $\thpda(G)=2\gamma(H)$.
Furthermore, any power dominating set for $G$ that is a subset of $V(H)$ 
must be a dominating set for $H$.
\end{thm}

\begin{thm}\label{half-ratio-iff} {\rm \cite{product-power-throt}} 
Let $G$  be a connected graph of order $n\ge 2$. Then  $\thpda(G)=\frac n 2$ if and only if $G=(H\circ K_1)\circ K_1$ for some connected graph $H$,  $G=C_4\circ K_1$, or $G=C_4$.   \end{thm}


  \section{Universal product throttling}\label{s:prod-univ}

Carlson introduced universal definitions for propagation time and sum throttling for a zero forcing parameter $Y$ in \cite{Josh1}; here we use simplified versions.  Recall that for $S\subseteq V(G)$,  $\ptx(G;S)$  is the least $p$ such that $S^{[p]}=V(G)$, or infinity if $S$ is not a $Y$-forcing set of $G$. For a positive integer $k$, $\ptx(G,k)=\min_{|S|=k}\ptx(G;S)$ and the sum throttling number for $Y$ is
\[\thx(G)=
\min_{Y(G) \le k\le n}(k+\ptx(G,k)).\]
In addition to (standard and PSD) zero forcing, these definitions apply immediately to power domination with $S^{[i]}$ defined to be $P^{[i]}(S)$.  We define $\pt_c(G,k)=\capt_k(G)$ so that the universal notation can also be used for Cops and Robbers. Using the two definitions given for product cop throttling and product power throttling as models, we have two universal definitions for product throttling that apply to the four processes: Cops and Robbers, power domination, standard zero forcing, and PSD zero forcing,
  \[\mbox{(with initial cost) }\ \thxx(G)=\min_{Y(G) \le k\le n}k(1+\ptx(G,k))\]
 and
 \[\mbox{(with no  initial cost) }\  \thxa(G)=\min_{Y(G) \le k< n}k\ptx(G,k),\]
where the definition of $\thxa(G)$ applies only to connected graphs of order at least two.  Notice that  the case $k=n$ is excluded for $\thxa(G)$.
 We also use the related notation $\thxx(G,k)=k(1+\ptx(G,k))$ and $\thxa(G,k)=k\ptx(G,k)$.

In this section, we record some simple consequences of these definitions, including possible low values of product throttling numbers.  But first  we make some comments based on prior results for $\thcx(G)$ and $\thpda(G)$. 
The nature of the application may motivate  the choice of definition (e.g., cops need to move to their positions whereas PMUs are fixed and immediately available).   It seems that $\thxx(G)$ favors a small number of cops/PMUs/blue vertices, whereas  $\thxa(G)$ seems to favor  a small propagation/capture time.  For example,  $\thcx(H)=\thcx(H,1)$ for $H$ chordal, whereas there are many graphs with $\thpda(G)=\gamma(G)$.  Note that $\gamma(G)$ is the smallest number of vertices that power dominate $G$ in one round. Just as we have defined $\ptx(G,k)$ to be the minimum possible propagation time using  $k$ vertices, it is useful to record the minimum number of vertices that can be used to achieve propagation time $p$ more generally:  Define $\kx(G,p)=\min\{k:\ptx(G,k)=p\}$. With this definition, we expect $\thxa(G)=\kx(G,1)$ to be common, although it may be written as $\gamma(G)$, because $k_{\pd}(G,1)=k_c(G,1)=\gamma(G)$.   

For the future,  it would be desirable to consider the effect of  initial cost values other than one in the case of product throttling with initial cost, i.e. to consider product throttling with with  variable initial cost
\[ \thxx(G,\omega)=\min_{Y(G) \le k\le n}k(\omega+\ptx(G,k)).\]

The results listed below assume that for all graphs $G$ of order $n$, we have  $1\le Y(G)\le n$,   $\ptx(G,n)=0$, and $S'\subseteq S$ implies $\ptx(G;S)\le\ptx(G;S')$; furthermore, $ Y(G)\le n-1$ and $\ptx(G,n-1)=1$ when  $G$ is connected and $n\ge 2$. 
 These conditions are satisfied by Cops and Robbers, power domination, and both standard and PSD zero forcing.

\subsection{General observations about $\thxx(G)$ and $\thxa(G)$} \label{ss:univ}

We begin with observations that involve  both definitions.

\begin{obs} For any connected graph $G$ of order at least two, $\thxx(G,k)=k+\thxa(G,k)$ and thus $\thxa(G)< \thxx(G)$.
  \end{obs}

 \begin{rem}\label{r:th-rad}  Let $Y$ be one of (standard or PSD) zero forcing, power domination, or Cops and Robbers. Since $\ptx(G,k)\ge \rad_k(G)$, \[\thxx(G)\ge  \min_{Y(G)\le k\le n}k(1+\rad_k(G))\mbox{ and }\thxa(G)\ge  \min_{Y(G)\le k< n}k\rad_k(G)\]
 ($G$ must be connected and of order at least two for $\thxa(G)$ to be defined).  \end{rem}

Next we consider $\thxx(G)$.

 \begin{obs}\label{o:bd-x} For any  graph $G$ of order $n$,  $Y(G)\le  \thxx(G)\le n$ since $\thxx(G,k)=\infty$ for $k<Y(G)$ and $\thxx(G,n)=n$.  If $G$ is connected and $n\ge 2$, then $Y(G)+1\le  \thxx(G)$, because $Y(G)\le n-1$.   \end{obs}

 \begin{obs}\label{o:ub-x} For every  graph $G$:
\ben[$(1)$]
\item $\thxx(G)\le Y(G)(1+\ptx(G))$.  
\item $\thxx(G)\le 2\kx(G,1)$.
\een
\end{obs}

  \begin{rem}  For any graph $G$ of order $n$,  $Y(G)\ge \frac n 2$ implies $\thxx(G)=n$: If $Y(G)\le |S|<n$, then $\ptx(G;S)\ge 1$ so $1+ \ptx(G;S)\ge 2$ and $|S|(1+ \ptx(G;S))\ge n$. \end{rem}

 The previous remark is not very useful for connected graphs when $Y(G)\le \gamma(G)$, including Cops and Robbers and power domination,  since $\gamma(G)\le \frac n 2$ for a connected graph of order $n\ge 2$ \cite[Theorem 2.1]{dom-book}.  
 However, it can be useful  for other parameters (and for all parameters when disconnected graphs are considered). 

Remark \ref{r:cr-sum=prod}, which relates sum and product cop throttling, is valid more generally for $\thxx(G)$.

\begin{rem}\label{r:sum=prod}
Let $G$ be a graph of order $n$.  Since $k\ge 1$,  
\[\thxx(G,k)=k(1+ \ptx(G,k))\ge k+\ptx(G,k)=\thx(G,k),\] so $  \thxx(G)\ge \thx(G)$.  Since we have equality exactly when $k=1$ or $\ptx(G,k)=0$, $\thxx(G)=\thx(G)$ if and only if $\thx(G)=\thx(G,1)$ or $\thx(G)=\thx(G,n)$.    
\end{rem}

 Finally we consider $\thxa(G)$.
 
   \begin{obs}\label{o:thxa-bds} For any connected graph $G$ of order $n\ge 2$,  $Y(G)\le  \thxa(G)\le \kx(G,1)\le n-1$.  Furthermore, $\thxa(G)\le Y(G)\ptx(G)$ and $ \thxa(G)=Y(G)$ if and only if $\ptx(G)=1$.
  \end{obs}

 The next remark could be stated more generally, but we present it in a form we find useful in characterizing graphs with high $\thxa(G)$.
 
 \begin{rem}\label{r:thxa-induced}  Let $Y$ be one of (standard or PSD) zero forcing, power domination, or Cops and Robbers. Suppose $G'$ is a connected graph  of order $n'\ge 2$ that is an induced subgraph of a connected graph $G$ of order $n$ such that $\thxa(G')=\kx(G',1)$.  Then there is a set $S'\subset V(G')$ such that $\thxa(G')=\thxa(G';S')$ and $\ptx(G';S')=1$. Define $S=S'\cup (V(G)\setminus V(G'))$.  Since $\ptx(G;S)\le \ptx(G';S')=1$, \[\thxa(G)\le  n-n'+\thxa(G').\]
  \end{rem}

For   $\thxa(G)$ there is a relationship with sum throttling.

\begin{prop}\label{p:thx-1} For any connected graph of order $n\ge 2$, $\thxa(G)\ge \thx(G)-1$.  If $\thx(G)=\thx(G,1)=1+\ptx(G,1)$ or $\thx(G)=\thx(G,\kx(G,1))=\kx(G,1)+1$, then $\thxa(G)=\thx(G)-1$.
\end{prop}
\bpf  Let $t=\thx(G)$. Then $\ptx(G,k)\ge t-k$ for all $k=1,\dots,t$. It is straightforward to verify that $k(t-k)\ge t-1$ for all $k=1,\dots,t-1$.  Thus $k\ptx(G,k)\ge \thx(G)-1$ for all $k=1,\dots,t-1$. Furthermore, $k\ptx(G,k)\ge t>\thx(G)-1$ for $k=t,\dots,n-1$.  Thus $\thxa(G)\ge \thx(G)-1$ since  $\thxa(G)=k\ptx(G,k)$ for some $k$ with $1\le k \le n-1$. 

 If $\thx(G)=\thx(G,1)$, then $\thxa(G)\le \thxa(G,1)=\thx(G)-1$.  The argument for $\thx(G)=\thxa(G,\kx(G,1))$ is similar. 
\epf

\subsection{Low values of the product throttling number}\label{s:univ-low}

For low values of $\thxx(G)$ or $\thxa(G)$, results can usually be described using graphs with low values of $Y(G)$ and $\ptx(G)$.  Recall that all graphs are connected and of order at least two.

\begin{rem}\label{r:univ-low} Let $G$ be a connected graph of order $n\ge 2$.  Setting $\thxx(G)=t$ or $\thxa(G)=t$  and factoring $t$ yields the following results for small $t$.
\ben[(a)]
\item\label{x1} No graph of order two or more has $\thxx(G)=1$. 
\item\label{a1} $\thxa(G)=1$ if and only if $Y(G)=1$ and $\ptx(G)=1$.
\item\label{x2} $\thxx(G)=2$ if and only if $Y(G)=1$ and $\ptx(G)=1$. 
\item\label{a2} $\thxa(G)=2$ if and only if 
$G$ satisfies exactly one of the following conditions:
		\ben[(i)]
			\item $Y(G)\le 2$, $\ptx(G,2)=1$, and $\ptx(G,1)> 2$.
			\item $Y(G)=1$ and $\ptx(G,1)=2$.
		\een
\item\label{x3} $\thxx(G)=3$ if and only if $G$ satisfies exactly 
one of the following conditions:  
		\ben[(i)]
			\item $n=3$ and $Y(G)>1$.
			\item $Y(G)=1$ and $\ptx(G,1)=2$.
		\een
\item\label{a3} $\thxa(G)=3$ if and only if
$G$ satisfies exactly  
one of the following conditions:
		\ben[(i)]
			\item $Y(G)\le 3$, $\ptx(G,3)=1$, $\ptx(G,2)>1$, and $\ptx(G,1)> 3$.
			\item $Y(G)=1$, $\ptx(G,1)=3$, and $\ptx(G,2)>1$.
		\een
\item\label{x4} $\thxx(G)=4$ if and only if  $G$ satisfies exactly one of the following conditions:
		\ben[(i)]
			\item $n=4$, $Y(G)>1$, and $\ptx(G,2)>1$.
			\item $Y(G)\le 2$, $\ptx(G,2)=1$, and $\ptx(G,1)> 3$.
			\item $Y(G)=1$ and $\ptx(G,1)=3$.
		\een
\een
For \eqref{x2}, note that $\thxx(G)=\thxx(G,2)$ implies $\ptx(G,2)=0$ and thus $G=K_2$, which is covered by $Y(G)=1$ and $\ptx(G)=1$.
 For \eqref{x3}(ii) and \eqref{x4}(ii), note that  $\ptx(G,1)\ge 2$ implies $n\ge 3$, so $\ptx(G,2)\ge 1$ and $\thxx(G,2)\ge 4$. 
\end{rem}

\begin{obs}
Let $p$ be  a  prime number. 
\ben[$(1)$]
\item If $\thxx(G)=p$, then $Y(G)=1$ and $\ptx(G)=p-1$, or $|V(G)|=p$. 
\item If $\thxa(G)=p$, then $Y(G)=1$ and $\ptx(G)=p$, or $\ptx(G,p)=1$. 
\een
\end{obs}

\section{Product throttling for standard zero forcing}\label{s:prod-Z}

For any set $S\subseteq V(G)$, at most $|S|$ forces can be performed in each round, so $\frac {n-|S|}{|S|}   \le  \pt(G;S)$.  This implies for $1\le k\le n$, \beq\frac {n-k}k   \le  \pt(G,k).\label{eq:Zbd}\eeq 
This fundamental bound is the reason that the initial cost version of
product throttling for  standard zero forcing, defined by $\thzx(G):=\min_{S\subseteq V(G)}|S| (1+ \pt(G;S))$, is not interesting.  

\begin{rem}\label{r:thzx-boring} For any graph $G$ of order $n$, $\min_{S\subseteq V(G)}|S| (1+ \pt(G;S))=n$, which is achieved by coloring all vertices blue, because  
 \[  k (1+ \pt(G,k))\ge k \lp1+ \frac {n-k}{k}\rp= n\]
 by \eqref{eq:Zbd}.
\end{rem}

\subsection{Characterization of $\thza(G)$}
Next we consider the version of product throttling for standard zero forcing that has no initial cost and show that $\thza(G):=\min_{\Z(G)\le k<n} k\pt(G,k)$ is the least $k$ such that $\pt(G,k)=1$.  First we need to define some terms  and prove a lemma.  Given a standard zero forcing set $S$ and a propagating set  of forces $\F$, a {\em forcing chain}  is a maximal sequence of vertices $(v_1,v_2,\dots,v_s)$ such that for $i=1,\dots,s-1$,  $v_i \to v_{i+1}$. The set of vertices of a forcing chain  necessarily induces a path. A \emph{reversal}   of $S$ is the set of last vertices of the zero forcing chains of a propagating set of forces $\F$, i.e., the vertices that do not perform forces.  Note that a set $S$ often has more than one set of forces, but a given propagating set of forces $\F$ has a unique set  of forcing chains that thus defines one reversal of $S$. Any reversal of $S$ can be denoted by $rev(S)$.
  Note that the cardinality of $rev(S)$ is the same as the cardinality of $S$.  
It is well-known that  if $S$ is  a  zero forcing set of $G$, then any reversal of $S$ is also a zero forcing set of $G$ \cite{proptime}.

 \begin{lem}\label{rev-lem} Let $G$ be a graph, $S$ a zero forcing set of $G$, and let $t=\pt(G;S)$. 
 Then for any reversal of $S$ and $i=0,\dots,t$, 
 \[  S^{(t-i)}\subseteq (rev(S))^{[i]}.\] 
 \end{lem}
 \bpf Let $\F$ be a propagating set of forces that produces $rev(S)$.
The result is established by induction on $i$.   For $i=0$, $S^{(t-0)}=S^{(t)}\subseteq rev(S)=(rev(S))^{[0]}$. Assume that  $S^{(t-j)}\subseteq (rev(S))^{[j]}$ for $0\le j <i$.  Let $w\in S^{(t-i)}$. If $w$ does not perform a force in $\F$, then $w\in rev(S)\subseteq (rev(S))^{[i]}$. So assume $w\to u$ in round $t-j>t-i$, so $j<i$.  By the induction hypothesis, $u\in (rev(S))^{[j]}\subseteq (rev(S))^{[i-1]}$. Now consider a neighbor $v\ne w$ of $u$.  Since $v$ does not force $w$ and $w\to u$ in round $t-j$, either $v$ does not force or $v\to x$ in round $t-j'$ with $j'<j$.  In either case, $v\in  (rev(S))^{[i-1]}$.  So if  $w\not\in  (rev(S))^{[i-1]}$, then $u$ can force $w$ in the $i$th round of forcing starting with $rev(S)$.  Thus $w\in  (rev(S))^{[i]}$.
\epf

As in Section \ref{ss:univ}, define $k(G,p)=\min\{|S|:\pt(G;S)=p\}$.  
\begin{thm}\label{t:thza-not-throttle}  For any graph $G$, $\thza(G)$ is the least $k$ such that $\pt(G,k)=1$, i.e., $\thza(G)=k(G,1)$. Necessarily  $k(G,1)\ge\frac n 2$.
\end{thm}
\bpf Let $k\ge \Z(G)$, let $t=\pt(G,k)$ and let  $S\subsetneq V(G)$ be such that $\pt(G;S)=t$.  If $\pt(G,k)=1$, there is nothing to prove for this $k$, so  suppose $t\ge 2$.  Define $\hat S=S\cup rev(S)$ for some reversal of $S$.  Then for $i=1,\dots,t$, 
\[S^{[i]}\cup \bigcup_{j=0}^{i}S^{(t-j)}\subseteq S^{[i]}\cup (rev(S))^{[i]} \subseteq \hat S^{[i]}\]  by Lemma \ref{rev-lem}.  In particular  (since $t\ge 2$),  
\[V(G)=\bigcup_{j=0}^t S^{(j)}\subseteq \hat S^{\scriptsize\lb\lc \frac{t-1}2\rc\rb}.\]  Thus $\pt(G,2k)\le \lc \frac{t-1}2\rc \le \frac t 2$ and $(2k)\pt(g,2k)\le k\pt(G,k)$.  Apply this repeatedly as needed to show that $\min_{\Z(G)\le k<n} k\pt(G,k)=\min \{k: \pt(G,k)=1\}=k(G,1)$.

If $\pt(G,k)=1$, then $1\ge \frac {n-k}k$, so $k\ge \frac n 2$. \epf

\begin{rem}\label{r:p1} For  a connected graph $G$   of order $n\ge 2$,  $k(G,1)\ge \gamma(G)$ since to achieve $\pt(G;S)=1$, every vertex must be in $S$ or adjacent to a vertex in $S$. However, $k(G,1)$ can be much larger than the domination number.  For example, $k(K_n,1)=n-1$ whereas $\gamma(K_n)=1$. \end{rem}

While the question of determining $k(G,1)=\thza(G)$  is  interesting, it seems more like a  question about a form of domination or zero forcing rather than a throttling question (since there is no balancing of resources and time). However, in the next two sections we  offer characterizations of extreme values of $\thza(G)$.

\subsection{Low values  of $\thza(G)$}

\begin{obs}\label{o:thza-throttle-method} 
By Theorem \ref{t:thza-not-throttle},
$\thza(G) = {\frac {n} 2}$ if and only if $G$ has a zero forcing set $S$ with $\pt(G;S)=1$ and $|S|={\frac {n} 2}$.  Analogously, in the case $n$ is odd, $\thza(G) = {\frac {n+1} 2}$ is equivalent to the existence of a zero forcing set $S$ such that $\pt(G;S)=1$ and $|S|={\frac {n+1} 2}$.
\end{obs}

Next we use Observation \ref{o:thza-throttle-method} to show that paths have the minimum possible product throttling number. 

 \begin{prop} \label{e:thza-Pn} For $n\ge 2$, 
 $\thza(P_n)= \begin{cases}
\frac n 2 & \mbox{if $n$ is even} \\
\frac {n+1} 2 & \mbox{if $n$ is odd } 
\end{cases}. $
 \end{prop}
\bpf Assume $V(P_n)=\{v_1,\ldots,v_{n}\}$ and $E(P_n)=\{v_iv_{i+1}:1\leq i\leq n-1\}$. Define $T_n\subset V(P_n)$ by
 \[T_n= \begin{cases}
\{v_i\in V(P_n): i \equiv 2\!\!\mod 4 \mbox{ or } i \equiv 3\!\!\mod 4 \}  & \mbox{if }n\not\equiv 1\mod 4\\
\{v_i\in V(P_n): i \equiv 2\!\!\mod 4 \mbox{ or } i \equiv 3\!\!\mod 4 \}\cup\{v_n\}  & \mbox{if }n \equiv 1\mod 4\end{cases}. \]
When $n$ is even,  
\[N(v_i)\setminus T_n=\begin{cases}
        v_{i-1} & \mbox{ if } i \equiv 2\!\!\mod 4\\
        v_{i+1} & \mbox{ if } i \equiv 3\!\!\mod 4\end{cases}\] for every $v_i\in T_n.$
As a result, each vertex in $T_n$ forces in a single round.  Now suppose  $n$ is odd. The same is true for all vertices  except  $v_n$, which has no neighbor outside $T_n$ if $n\equiv 3\mod n$ and which shares a neighbor with $n_{n-2}$ if $n\equiv 3\mod n$.  Thus propagation time is again one and  $\thza(P_n)=|T_n|=\frac{n+1}2$.  \epf

The next result is a direct consequence of Theorem \ref{t:thza-not-throttle} and Proposition  \ref{e:thza-Pn}.

 \begin{cor} \label{c:thza-lowbound} 
For any graph connected graph $G$ of order $n$, $\thza(G)\geq \Big\lceil {\frac {n} 2}\Big\rceil$ and this bound is tight. 
 \end{cor}

A characterization of all graphs attaining the lower bound in Corollary \ref{c:thza-lowbound} will be presented after introducing some terminology and obtaining a few preliminary results. We start by recalling the definition of matched-sum graphs, as introduced by Georges and Mauro in \cite{GM02}:
Let $G_1$ and $G_2$ be graphs and let $M$ be a matching between $V(G_1)$ and $V(G_2)$. Then the {\em{$M$-matched-sum of $G_1$ and $G_2$}} (or simply, {\em {$M$-sum of $G_1$ and $G_2$}}), denoted by $G_1M^+G_2$, is the graph with $V(G_1M^+G_2)=V(G_1)\cup V(G_2)$ and $E(G_1M^+G_2)=E(G_1)\cup E(G_2)\cup M$. A graph of the form $G_1M^+G_2$ is also called a \emph{matched-sum graph}.

\begin{ex} \label{e:thza-matchedPn}
A path of even order is an example of a matched-sum graph. We show this by identifying graphs $G_1$, $G_2$ and $M$ such that $G_1M^+G_2$ is isomorphic to $P_n$, for any even integer $n$.  Let $G_1$ be the subgraph of $P_n$ induced by the corresponding zero forcing set $T_n$, as defined in Proposition \ref{e:thza-Pn}. Let $G_2$ be the subgraph of $P_n$ induced by $V(P_n)\setminus T_n$, and let $M$ be the matching defined by the edges used in the propagation where each vertex of $G_1$ forces its only neighbor in $G_2$.  It is straightforward to verify that $G_1M^+G_2=P_n$.  See Figure \ref{fig:path-match}.\vspace{-3pt}
\end{ex}

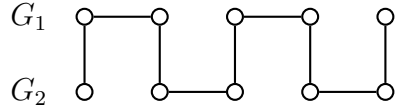
\begin{figure}[h!] \begin{center}
\begin{tikzpicture}[scale = 1]
    \vertex[color=white,label = {right:$G_1$}] (G1) at (-1.25,1){};
    \vertex[color=white,label = {right:$G_2$}] (G2) at (-1.25,0){};
    \foreach \n in {0,...,4}{
        \vertex (a\n) at (\n,0){};
        \vertex (b\n) at (\n,1){}; }
    \foreach \n in {0,...,4}{
        \draw[thick] (a\n) to (b\n);}
    \foreach \m/\n in {1/2,3/4}{
        \draw[thick] (a\m) to (a\n);}
    \foreach \m/\n in {0/1,2/3}{
        \draw[thick] (b\m) to (b\n);}
\end{tikzpicture}
\caption{$P_{10}$ drawn as a matched-sum graph. \vspace{-12pt} \label{fig:path-match}}
\end{center}
\end{figure}
\begin{obs} 
For any two connected graphs $G_1$ and $G_2$, a graph $G_1M^+G_2$ is connected for any matching $M$. However, it is not necessary for $G_1$ and $G_2$ to be connected, to obtain a connected graph matched-sum graph $G_1M^+G_2$ (cf. Example \ref{e:thza-matchedPn}).  
\end{obs}

In our work we assume a matched-sum graph $G_1M^+G_2$ is connected but we do not require $G_1$ and $G_2$ to be connected.  Readers interested in details on the connectivity of matched-sum graphs are referred to \cite{BGVM08}.

\begin{obs} Any matched-sum graph $G_1M^+G_2$ necessarily has even order. 
\end{obs}

  As shown in Proposition \ref{e:thza-Pn}, it is possible for a graph $G$ of odd order $n$ to satisfy  $\thza(G) = \Big\lceil {\frac {n} 2}\Big\rceil$. As a result, matched-sum graphs are not sufficient to characterize all graphs $G$ attaining the lower for $\thza(G)$ in Corollary \ref{c:thza-lowbound}, but they capture essential structural properties necessaries for a graph to have a zero forcing set with propagation time $1$ of minimum size.\\

We note that matched-sum graphs have also been studied under other names, such as $(G_1,G_2)$-{\em permutation graphs} \cite{BB15} or {\em matching graphs} \cite{proptime}.  In particular, matched-sum graphs were used by Hogben et al. \cite{proptime} to study graphs having a minimum zero forcing set of propagation time $1$. We recall a result from \cite{proptime} with a particular connection to the work in this section.

\begin{prop}\label{p:thza-throttle-proptime}{\rm \cite{proptime}} Let $G$ be a graph of order $n$. Then any two of the following conditions imply the third.
\ben[$(1)$]
\item $n = 2\Z(G)$.
\item $\pt(G) = 1$.
\item $G$ is a matched-sum graph.
\een
\end{prop}

Let $G$ be a graph of even order $n$. Observe that Conditions 1 and 2 in Proposition \ref{p:thza-throttle-proptime} directly imply $\thza(G)= {\frac {n} 2}= \Z(G)$ and, by Proposition \ref{p:thza-throttle-proptime}, $G$ is a matched-sum graph. As shown in Example \ref{e:thza-matchedPn}, the condition $ {\frac {n} 2}= \Z(G)$  is not necessary to obtain $\thza(G)= {\frac {n} 2}$.  In our next result, we prove that the condition $\thza(G)= {\frac {n} 2}$ is both necessary and sufficient for $G$ to be a matched-sum graph. 

\begin{thm}\label{t:thza-throttle=halfeven}  
A connected graph $G$ of even order $n$ satisfies $\thza(G)= {\frac {n} 2}$ if and only if $G$ is a connected matched-sum graph.
\end{thm}

\bpf
Suppose $G$ is a matched-sum graph. Let $H$ and $H'$ be a pair of vertex-disjoint graphs, and let $M$ be a matching between $V(H)$ and $V(H')$ such that $G=HM^+H'$. 
It is immediate that $V(H)$ is a zero forcing set of $G$ with order ${\frac {n} 2}$ such that $\pt(G, V(H))=1$, which implies  $\thza(G)={\frac {n} 2}$ by Observation \ref{o:thza-throttle-method}.

Let $G$ be a connected graph of even order $n=2r$ such that $\thza(G)= r$. By Observation \ref{o:thza-throttle-method}, there exists $S\subset V(G)$ zero forcing set of $G$ with $|S|=r$ and $\pt(G;S)=1$. 
Let $H$ and $H'$ be the subgraphs of $G$ induced by $S$ and $V(G)\setminus S$, respectively. Let $M$ be the set of edges $uv$ where $u\to v$; we show $G=HM^+H'$. It is immediate that $H$ and $H'$ are vertex-disjoint graphs of order $r$ satisfying $V(G)=V(HM^+H')$. By definition of matched-sum graph, $E(HM^+H')=E(H)\cup M\cup E(H')$, and the selection of $M$, $H$ and $H'$ guarantees $E(HM^+H') \subseteq E(G)$. To conclude the proof, it is sufficient to show $xy\in E(G)$ with $x\in S$ and $y\not\in S$ implies $xy\in M$.  
Suppose to the contrary that $xy\not\in M$. Then there exists $w\not\in S$ such that $xw\in M$.  But then $x$ has two neighbors in $V(G)\setminus S$, so $\pt(G;S)>1$, which is a contradiction.
\epf

The characterization of all graphs $G$ of odd order $n$ such that $\thza(G)= \lc {\frac {n} 2}\rc$ is given in terms of the a graph operation we recall next. If $G$ is a graph of order $n\geq 2$ and $v$ is a vertex of $G$, then the graph $G-v$ is defined by $V(G-v)=V(G)\setminus \{v\}$ and $E(G-v)=E(G)\setminus \{uv: u\in N_G(v)\}$. That is, $G-v$, the graph obtained by removing the vertex $v$ in all edges incident with $v$ in $G$.

\begin{thm}\label{t:thza-throttle=halfodd}  
A connected graph $G$ of odd order $n$ satisfies $\thza(G)= \lc {\frac {n} 2}\rc$ if and only if there exists $v\in V(G)$ such that $G-v$ is a matched-sum graph.
\end{thm}

\bpf
Suppose $v\in V(G)$ and $G-v$ is a matched-sum graph. Then by Theorem \ref{t:thza-throttle=halfeven}, there is a set $S\subset V(G-v)$ such that $|S|=\frac{n-1}2$ and $\pt(G-v;S)=1 $.  Define $S'=S\cup\{v\}$.  Then $\frac{n+1}2=|S'|=\thza(G;S)\ge \thza(G)$, and $\thza(G)= \lc {\frac {n} 2}\rc$ by Theorem \ref{t:thza-not-throttle}.

Let $G$ be a graph of odd order $n=2r+1$  such that $\thza(G)= r+1$. By Observation \ref{o:thza-throttle-method}, there exists $S\subset V(G)$ with $|S|=r+1$ such that $S$ is a zero forcing set of $G$ and $\pt(G;S)=1$. Let $\F$ be a propagating set of forces and  let $v\in S$ be the one vertex that does not perform a force.  Then $S'=S\setminus \{v\}$ is zero forcing set for $G-v$ and $\pt(G-v;S')=1$.  Thus $\thza(G-v)=r=\frac{|V(G-v)|}2$, so $G-v$ is a matched-sum graph by Theorem  \ref{t:thza-throttle=halfeven}.
\epf

Matched-sum graphs include several interesting graph families, and for these graphs the product zero forcing  throttling number is obtained by applying Theorem \ref{t:thza-throttle=halfeven}. For example, the {\em $d$-dimensional hypercube} $Q_d=K_2\cp \cdots\cp K_2$ (with $d$ copies of $K_2$) is a matched-sum graph for any $d\geq 2$, so Theorem \ref{t:thza-throttle=halfeven} yields $\thza(Q_d)=2^{d-1}$.  The generalized Petersen graph $G(r,s)$ is also a matched-sum graph, so $\thza(G(r,s))=r$.

\subsection{High values  of $\thza(G)$}

The maximum value of $\thza(G)$ over connected graphs $G$ of order $n\ge 2$ is $n-1$ (see Observation \ref{o:thxa-bds}) and this is realized by $K_n$.  We can use Carlson and Kritschgau's  characterization of graphs having $\thz(G)=n$ in \cite{JK-forbid} to characterize graphs having $\thza(G)=n-1$.

\begin{figure}[h!] \begin{center}
\begin{tikzpicture}[scale = .75]
    \vertex (v0) at (0,0){};
    \vertex (v1) at (-1.5,1){};
    \vertex (v2) at (-1.5,-1){};
    \vertex (v3) at (1.5,1){};
    \vertex (v4) at (1.5,-1){};
    \foreach \m/\n in {0/1,0/2,0/3,0/4,1/2,3/4}{
        \draw[thick] (v\m) to (v\n);}
\end{tikzpicture}
\caption{The bowtie graph. \label{fig:bowtie}}\vspace{-15pt}
\end{center}
\end{figure}
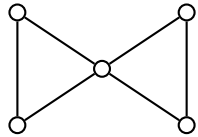

\begin{thm}\label{t:thz-hi-ext}{\rm \cite{JK-forbid}} For  a connected graph $G$   of order $n$,  $\thz(G)=n $ if and only if  $G$ does not have  a $P_4, C_4$, or  bowtie graph (see Figure \ref{fig:bowtie}) as an induced subgraph. 
\end{thm}

\begin{cor}\label{c:thza-hi-ext} For  a connected graph $G$   of order $n\ge 2$,  $\thza(G)=n-1 $ if and only if  $G$ does not have a $P_4, C_4$, or  bowtie graph as an induced subgraph. 
\end{cor}
\bpf Since $\thza(G)\le n-1$, it suffices to establish that $\thza(G)\le n-2 $ if and only if  $G$ has a $P_4, C_4$, or  bowtie graph as an induced subgraph.

Let $G'$ be one of $P_4, C_4,$ or the bowtie graph and let $n'$ be the order of $G'$. Observe that $\pt(G',n'-2)=1$, so $\thza(G')\le n'-2$.  Suppose that a  connected graph $G$ of order $n$ contains $G'$ as an induced subgraph.
  Then by Remark \ref{r:thxa-induced}, $\thza(G)\le n-n'+(n'-2)=n-2 $.

Now suppose $G$ is a connected graph of order $n\ge 2$ and $\thza(G)\le n-2$. Then $\thz(G)\le \thza(G)+1=n-1$ by Proposition \ref{p:thx-1},  so   $G$ has a  $P_4, C_4$, or bowtie graph as an induced subgraph by Theorem \ref{t:thz-hi-ext}.
\epf


\section{Product throttling  for Cops and Robbers revisited}\label{s:prod-CR-a}
In Section \ref{s:prod-CR}, we summarized known results on the product cop throttling number with initial cost, $\thcx(G)$. In this section, we study product cop throttling number with no initial cost.  Let $G$ be a connected graph of order at least two. For a set $S\subseteq V(G)$ with $c(G)\le|S|\leq \gamma(G)$, $\thca(G;S)=|S|\capt(G;S)$. Define $\thca(G,k)=k\capt_k(G)$. The {\em no initial cost product cop throttling number} of $G$ is 
 \[\thca(G)=\min_{c(G) \le k\le\gamma(G)}k\capt_k(G)=\min_{c(G) \le k\le\gamma(G)}\thca(G,k).\]
While initial cost product throttling seems more realistic if considering actual police officers, no initial cost product throttling is useful in other searching applications. This parameter has been studied as {\em work} $w_k=k\capt_k(G)$ and {\em speedup} between using $j>i$ cops,  defined as $w_i/w_j$ \cite{LP16,LP17,LP19}. The  product cop power throttling number with no initial cost extends this idea by considering the number of cops that yields the largest possible speed-up.
We give bounds for $\thca(G)$ and determine this number exactly for certain families of graphs, including paths, cycles, complete graphs, complete bipartite graphs, full $t$-ary trees and unit interval graphs.  We also establish a few additional results for $\thcx(G)$.

\subsection{General observations}

 \begin{obs} \label{o:thca-bds} Let $G$ be a connected graph of order at least two.
There are several immediate upper bounds for $\thca(G)$:
\ben[$(1)$]
\item $\thca(G)\le \gamma(G)$.
\item $\thca(G)\le c(G)\capt(G)$.
\item $\thca(G) \le  \thcx(G)-c(G)$.\smallskip

\hspace{-10mm} There are several immediate lower bounds for $\thca(G)$:
\item $ \thca(G)\ge c(G)$.
\item $ \thca(G)  \ge \thc(G)-1$.
\item $ \thca(G)\ge  \min_{c(G)\le k\le \gamma(G)}k\rad_k(G)$.
\een
Furthermore, $\thca(G)= c(G)$ if and only $\capt(G)=1$.
\end{obs}

\begin{rem}\label{thca:vals} The following values for $K_n$ and $K_{r,n-r}$ follow immediately  from Observation~\ref{o:thca-bds}.  Capture time agrees with power propagation time for paths and cycles, so the values of these graphs follow from Observation~\ref{p:pathcycle}. 
\ben[$(1)$]
\item \label{thcaKn}  $\thca(K_n)=1=\gamma(K_n)$. 
\item \label{thcaKpq} For $2\le r\le n-2$, $\thca(K_{r,n-r}) = 2=\gamma(K_{r,n-r})$ and $\thca(K_{1,n-1}) = 1=\gamma(K_{1,n-1})$. 
\item \label{thcaPn} $\thca(P_n)=\lc\frac n 3\rc=\gamma(P_n)$.
\item $\thca(C_n)=\lc\frac n 3\rc=\gamma(C_n)$ for $n\geq 4$. 
\een
\end{rem}

The next result is immediate from Remark \ref{r:univ-low} and Proposition \ref{p:thcx-low}.

\begin{rem}\label{r:c2-low} Let $G$ be a connected graph of order at least two.

\ben[$(1)$]
\item \label{thca1} $\thca(G)=1$ if and only if $\gamma(G)=1$. 
\item \label{thca2} $\thca(G)=2$ if and only if 
$G$ satisfies at least one of the following conditions:
		\ben[(a)]
			\item $\gamma(G) = 2$.			
			\item \label{2b} $\gamma(G) \geq 3$ and there exists $z \in V(G)$ such that 
			\ben[(i)] 
			\item for all $v \in V(G)$, $\dist(z, v) \leq 2$, and 
			\item \label{2ii} for all $w \in V(G) \setminus N[z]$, there is a vertex $u \in N[z]$ such that $N[w] \subset N[u]$.
			\een
		\een
\item \label{thca3} $\thca(G)=3$ if and only if
$G$ satisfies at least one of the following conditions:
		\ben[(a)]
			\item $\gamma(G) = 3$ and $\thca (G)\neq 2$.
			\item $c(G)=1$, $\capt(G,1)=3$, and $\capt(G,2)\ge 2$. 	
		\een
\een
\end{rem}

Note that whether or not $\capt(G,2)\ge 2$ can be determined by a polynomial time algorithm  (see \cite{HM06} or \cite[Algorithm 2]{cop-throttle}).

Let $G_1$ and $G_2$ be graphs such that $G_1\cap G_2= K_m$  for some $m$, and $G_1,G_2\neq K_m$. Then $G_1\cup G_2$ is the \emph{clique sum} of $G_1$ and $G_2$. Next we state bounds for the (sum) cop throttling number for clique sums and establish analogous bounds for the product cop throttling numbers of clique sums.

\begin{thm}{\rm\cite{cop-throttle}}
 Let $G$ be a clique sum of $G_1$ and $G_2$. Let $k_1$ and $k_2$  be numbers such that $\thc(G_i)=\thc(G_i,k_i)$ for $i\in\{1,2\}$. Then \[\max\{\thc(G_1,k_1),\thc(G_2,k_2)\}\leq \thc(G)\leq k_1+k_2+\max\{\capt_{k_1}(G_1),\capt_{k_2}(G_2)\}.\]
\end{thm}

\begin{prop} Let $G$ be a connected non-trivial clique sum of $G_1$ and $G_2$. Let $k_1$, $k_2$, $\ell_1,\ell_2$ be such that for $i\in \{1,2\}$, $\thca(G_i)=k_i\cdot \capt_{k_i}(G_i)$ and $\thcx(G_i)=\ell_i \cdot (\capt_{\ell_i}(G_i)+1)$. Then
\bea\max\{\thca(G_1),\thca(G_2)\}&\leq &\thca(G)\leq (k_1+k_2)\max\{\capt_{k_1}(G_1),\capt_{k_2}(G_2)\},\text{ and} \\
 \max\{\thcx(G_1),\thcx(G_2)\}&\leq &\thcx(G)\leq (\ell_1+\ell_2)(\max\{\capt_{\ell_1}(G_1),\capt_{\ell_2}(G_2)\}+1). 
\eea
\end{prop}

\begin{proof}
 Let $k=k_1+k_2$. As in \cite{cop-throttle}, note that $G_1$ and $G_2$ are retracts of $G$, so $\capt_k(G)\leq \max\{\capt_{k_1}(G_1),\capt_{k_2}(G_2)\}$ (see 
 \cite{BPPR17}). Therefore \[\thca(G)\leq (k_1+k_2)\cdot\max\{\capt_{k_1}(G_1),\capt_{k_2}(G_2)\}.\]

For the lower bound, as in \cite{cop-throttle}, note that for $i\in\{1,2\}$, if a robber's movement within $G$ is restricted to $V(G_i)$, then for $k_i$ cops, there is no benefit to the cops starting outside $V(G_i)$. The $k_i$ cops then catch the robber in time $\capt_{k_i}{G_i}$.

The proof for $\thcx(G)$ is similar.
\end{proof}

\subsection{Chordal graphs}

Recall that a graph $G$ is a {\em chordal graph} if it has no induced cycle of length greater than 3. 
The next result is immediate from Theorem  \ref{t:thcx-chord} and Observation \ref{o:thca-bds}.
\begin{rem}\label{r:thca-chordal} Let $H$ be a connected chordal graph of order $n\ge 2$.  Then
\[\thca(H)=  \min_{1\le k\le \gamma(H)}k\rad_k(H)\le\min\{ \rad(H),\gamma(H)\}.\]
\end{rem}

Each of the upper bounds $ \rad(G)$ and $\gamma(G)$ is tight, but the inequality in Remark \ref{r:thca-chordal} cannot be changed to an equality. In Table \ref{tab-ex-thca}, we provide examples of graphs such that $\thca(G)=  \rad(G)$ and $\thca(G)= \gamma(G)$ as well as an example of a graph where $\thca(G) < \min\{ \rad(G),\gamma(G)\}$.

\begin{table}
\begin{center}
 \begin{tabular}{|c ||c |c |c|}\hline
$G$ &                       $\rad(G)$ & $\gamma(G)$ & $\thca(G)$\\\hline\hline
\begin{tikzpicture}[scale =.4]
 \foreach \x in {1,...,8}{\vertex (D\x) at (\x,-3) {};}
 \foreach \x in {1,3,5,7}{\vertex (C\x) at (\x+.5,-2) {};}
  \foreach \x in {2,6}{\vertex (B\x) at (\x+.5,-1) {};}
  \vertex (A) at (4.5,0) {};
   \draw[thick] (B6) to (A) to (B2);
    \draw[thick]  (C1) to (B2) to (C3);
    \draw[thick]  (C5) to (B6) to (C7);
    
    \draw[thick]  (D1) to (C1) to (D2);
    \draw[thick]  (D3) to (C3) to (D4);
    \draw[thick]  (D5) to (C5) to (D6);
    \draw[thick]  (D7) to (C7) to (D8);
\end{tikzpicture}   &$3$ & $ 5$ & $3=\rad(G)<\gamma(G)$ \\ \hline

\begin{tikzpicture}[scale =.4]
 \foreach \n in {1,...,9}{\vertex (x\n) at (\n,0) {};}
 \foreach \m/\n in {1/2,2/3,3/4,4/5,5/6,6/7,7/8,8/9}{
   \draw[thick]  (x\m) to (x\n);}
\end{tikzpicture} &$4$ & $3$ & $3=\gamma(G)<\rad(G)$ \\ \hline

\begin{tikzpicture}[scale =.4]
 \vertex (A) at (0,0) {};
 \foreach \x in {1,...,7}{
     \vertex (B\x) at (\x,.3*\x) {};
     \vertex (C\x) at (\x,-.3*\x) {};
     }
  \foreach \x in {1,2}{
     \vertex (D\x) at (-\x,.3*\x) {};
     \vertex (E\x) at (-\x,-.3*\x) {};
     \vertex (F\x) at (0,\x) {};
     \vertex (G\x) at (0,-\x) {};
     }    
  \foreach \m/\n in {2/3,3/4,4/5,5/6,6/7}{
      \draw[thick]  (B\m) to (B\n);
      \draw[thick]  (C\m) to (C\n);
 }
    \foreach \n in {B,C,D,E,F,G}{
      \draw[thick]  (A) to (\n1);
       \draw[thick]  (\n1) to (\n2);
 }
\end{tikzpicture}   &$7$ & $ 9$ & $6< \min\{\rad(G),\gamma(G)\}$ \\ \hline

\end{tabular}
\caption{Examples illustrating relationships among  $\rad(G)$, $\gamma(G)$, and $\thca(G)$.  \label{tab-ex-thca}}\vspace{-10pt}
\end{center}
\end{table}

 We can find $\thca(G)$ exactly for certain families of chordal graphs, including split graphs, full $t$-ary trees, and unit interval graphs.
A {\em split graph} is a graph whose vertices can be partitioned into a clique and an independent set. 

\begin{rem} 
Let $G$ be a connected split graph of order two or more. If $\gamma(G) = 1$, then  $\thca(G)=1$ by Remark \ref{r:c2-low}. If $\gamma(G) > 1$, then $\thca(G)=2$ by Remark~\ref{r:thca-chordal} since $\rad(G) = 2$.
\end{rem}

A \emph{full} $r$-ary tree of height $h$,  denoted by $T_{r,h}$, is a rooted tree in which each node has $r$ children unless it is at distance $h$ from the root, and the distance between a vertex and the root is at most $h$; $T_{3,2}$ is shown in Figure \ref{fig:T32}.  

   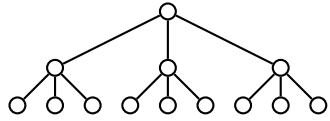
\begin{figure}[!h!]
\begin{center}
\begin{tikzpicture}[scale =.5]
     \foreach \n in {1,2,3,4,5,6,7,8,9}{\vertex (C\n) at (\n,-2) {};}
     \foreach \n in {2,5,8}{\vertex (B\n) at (\n,-1) {};}
     \vertex (A) at (5,.5) {};
     \foreach \n in {2,5,8}{ \draw[thick] (A) to (B\n);}
     \foreach \n in {1,2,3}{ \draw[thick] (B2) to (C\n);}
     \foreach \n in {4,5,6}{ \draw[thick] (B5) to (C\n);}
     \foreach \n in {7,8,9}{ \draw[thick] (B8) to (C\n);}
\end{tikzpicture}
\caption{The full ternary tree of height two $T_{3,2}$. \label{fig:T32}}\vspace{-12pt}
\end{center}
\end{figure}

\begin{prop}\label{p:r-ary-tree}  Let $T_{r,h}$ be the full $r$-ary tree of height $h$ with $h,r\geq 2$. Then $\thca(T_{r,h})=h$.
\end{prop}
\begin{proof}
Consider  a set $S\subseteq V(T_{r,h})$ of cardinality $k\geq 2$ such that $\capt(T_{r,h};S)=\capt(T_{r,h},k)$. If $k\capt(T_{r,h};S ) \leq h$, then each of the leaves must be within distance $\frac{h}{k}$ of a vertex in $S$. This requires at least $r^{h - \frac{h}{k}}\geq r^{\frac{h}{2}}$ vertices; it is most efficiently done by taking every vertex at depth $h-\lf\frac{h}{k}\rf$, so 
\[k \capt(T_{r,h};S ) \geq 2\cdot 2^{\frac{h}{2}} > h.\qedhere\]
\end{proof}

\begin{figure}[!h]
\centering
\begin{tikzpicture}[scale=.8]
\begin{scope}[very thick, every node/.style={sloped,allow upside down}]
\vertex (A) at (0,0){};
\vertex (B) at (-1.5, -0.75) {};
\vertex (C) at (-3, -1.55){};
\vertex (D) at (1.5, -0.75){};
\vertex (E) at (3,-1.55){};
\vertex (F) at (0,1.75){};
\draw(A) to (B);
\draw(A) to (D);
\draw(A) to (F);
\draw(C) to (B);
\draw(D) to (E);
\end{scope}
\end{tikzpicture}\\
\caption{An interval graph  $T$ with $\thca(T)<\gamma(T)$.    \label{f:interval}}
\end{figure}

As shown in \cite{product-power-throt}, the graph $T$ in Figure \ref{f:interval} is an interval graph (and chordal).  As with power domination, $\thca(T)=2<3=\gamma(T)$. 
We show that $\thpda(G)=\gamma(G)$ for a unit interval graph $G$; some of the ideas come from  the proof that $\thpda(G)=\gamma(G)$ in \cite{product-power-throt}, but there are substantial differences.  Instead of partitioning the vertices by the round in which they are observed, we partition the vertices of $G$ by their distance from $S\subset V(G)$: Define $S^{(k)}=\{v:\dist(S,v)=k\}$.  Then $V(G)=S\du S^{(1)}\du \dots\du S^{(\ecc(S))}$.
For a unit interval graph $G$, fix a unit representation of  $G$ with induced order $<$.  
For any vertex $v$, define $L(v)$ to be the least vertex in $N[v]$, $L^1(v)=L(v)$ and $L^{k+1}(v)=L(L^k(v))$.  Define  $R(v)$ to be the greatest vertex in $N[v]$, and define $R^k(v)$ analogously.

\begin{lem}
 \label{unit-int-lem-c}
 Let $G$ be a connected unit interval graph of order at least two with a fixed unit representation and induced order, and let   $S=\{x\}\subset V(G)$.  For $k=1,\dots,\ecc(S)$, $\{L^{k}(x),R^{k}(x)\}$ dominates $S^{(k)}\cup S^{(k+1)}$, and $L^{1}(x)$ dominates $x$. 
 \end{lem}

\begin{proof}
It is is immediate that $x$ is dominated by $L^{1}(x)$.
Let $v\in S^{(i)}$ with $i\ge 1$, so there is a path $(x=v_0,v_1,\dots, v_i=v)$ and $v_j\in  S^{(j)}$ for $j=0,\dots,i$. 
Suppose $v\le x$, so $v=v_i<\dots<v_1<x=v_0$. From the definition of $L^k(x)$, $(x=v_0,L^1(x),\dots,L^{i-1}(x), v_i=v)$ is also a path from $x$ to $v$.  By setting $i=k+1$,  we see that $L^{k}(x)$ dominates $v$. Now suppose $i=k$. Then $v$ is a neighbor of $L^{k-1}(x)$ so $L^{k}(x)\le v < L^{k-1}(x)$.  Since neighborhoods are consecutive and  $L^{k-1}(x)\in N[L^{k}(x)]$, $v \in N[L^{k}(x)]$.
Thus $L^{k}(x)$ dominates all  $v\in  S^{(k)}\cup S^{(k+1)}$ such that $v<x$.  The case of $v>x$ is handled by $R^{k}(x)$.
\end{proof}

 \begin{thm}
 \label{t:unit-int-c}
 If $G$ is a unit interval graph, then $\thca(G) = \gamma(G)$.
 \end{thm}
 
 \begin{proof}  
Let $G$ be a connected unit interval graph of order at least two with a fixed unit representation and induced order. 
It suffices to show $\gamma(G) \le   \thcx(G,k)  $ for $1\le k<\gamma(G)$.  Let  $t=\rad_k(G)$, and choose $S\subset V(G)$ such that $|S|=k$ and $\ecc(S)=t$; note that $t\ge 2$.  Define $T^k(S)=\cup_{x\in S}\{L^{k}(x),R^{k}(x)\}$. We consider two cases,  $t$ is even and $t$ is odd.

Assume first that $t$ is even. Then $\hat S=T^1(S) \cup T^3(S) \cup \dots \cup T^{t-1}(S)$ dominates $V(G)=S\cup S^{(1)}\cup \dots\cup S^{(t)}$  by Lemma~\ref{unit-int-lem-c}.  Then $\gamma(G)\le |\hat S|\le |S|2\frac t 2=kt=\thcx(G,k)$.

Now assume that $t$ is odd.  Let $\hat S = S \cup S^{(2)} \cup S^{(4)} \cup \dots \cup S^{(t-1)}$.   Then $\hat S=S\cup T^2(S) \cup T^4(S) \cup \dots \cup T^{t-1}(S)$ dominates $V(G)=S\cup S^{(1)}\cup \dots\cup S^{(t)}$ by  Lemma~\ref{unit-int-lem-c} and since the vertices in $S^{(1)}$ are dominated by   $S$ by definition. Then $\gamma(G)\le |\hat S|\le |S|\lp1+2\frac {t-1} 2\rp=kt=\thcx(G,k)$.
  \end{proof}


\section{Product throttling  for power domination revisited}\label{s:prod-power-2}
In this section we determine  $\thpda(G)$, the product power throttling number with no initial cost, for additional families of graphs $G$ and explore the definition of product throttling with initial cost,  $\thpdx(G)$.

\subsection{Determination of $\thpda(G)$ for additional families of graphs}\label{sfamilies}
$\null$ 

The values of  $\thpda(G)$ for some families of graphs were established in \cite{product-power-throt} and several families of graphs for which $\thpda(G)=\gamma(G)$ were presented (see Section \ref{s:prod-power-a}).   In this section, we establish $\thpda(G)$ for  some additional families of graphs.
We also construct  infinite families of graphs where $\thpda(G)\ne \gamma(G)$. For these families, $\thpda(G)<\min\{\gamma(G),\pd(G)\ppt(G)\}$.

We need the following definitions.  
For $j\ge 2$ and $d\ge 4$, construct the {\em $j,d$-generalized necklace $N_{j,d}$} by connecting $j$ copies of $K_d-e$ arranged cyclically to create a $d-1$ regular graph; $N_{3,5}$ is shown in Figure \ref{f:N35}.

\begin{figure}[h!]
\centering
\begin{tikzpicture}[scale=.5]
		\vertex (A) at (0,0) {};
		\vertex (B) at (1.25, 0) {};
		\vertex (C) at (2, 1.2) {};
		\vertex (D) at (.625,2) {};
		\vertex (E) at (-.625, 1.2) {};
		
		\vertex (F) at (3,4.5) {};
		\vertex (G) at (4.25, 4.5) {};
		\vertex (H) at (5, 5.7) {};
		\vertex (I) at (3.625,6.5) {};
		\vertex (J) at (2.375, 5.7) {};
		
		\vertex (K) at (-2.25,4.5) {};
		\vertex (L) at (-3.5, 4.5) {};
		\vertex (M) at (-4.25, 5.7) {};
		\vertex (N) at (-2.875,6.5) {};
		\vertex (O) at (-1.5, 5.7) {};
		
		\draw[thick] (A) to (C); 
		\draw[thick] (E) to (B); 
		\draw[thick] (A) to (D); 
		\draw[thick] (B) to (D); 
		
		\draw[thick] (F) to (H); 
		\draw[thick] (F) to (I); 
		\draw[thick] (I) to (G); 
		\draw[thick] (H) to (J); 
		
		\draw[thick] (K) to (M); 
		\draw[thick] (K) to (N); 
		\draw[thick] (L) to (N); 
		\draw[thick] (M) to (O); 
		
		\draw[thick] (A) to (B); 
		\draw[thick] (B) to (C); 
		\draw[thick] (C) to (D); 
		\draw[thick] (D) to (E); 
		\draw[thick] (E) to (A); 
		\draw[thick] (F) to (G); 
		\draw[thick] (G) to (H); 
		\draw[thick] (H) to (I); 
		\draw[thick] (I) to (J); 
		\draw[thick] (J) to (F); 
		\draw[thick] (K) to (L); 
		\draw[thick] (L) to (M); 
		\draw[thick] (M) to (N); 
		\draw[thick] (N) to (O); 
		\draw[thick] (O) to (K); 
		\draw[thick] (L) to (E); 
		\draw[thick] (G) to (C); 
		\draw[thick] (J) to (O); 
	\end{tikzpicture}

\caption{The 3,5-generalized necklace $N_{3,5}$.  \label{f:N35}}\vspace{-8pt}
\end{figure}
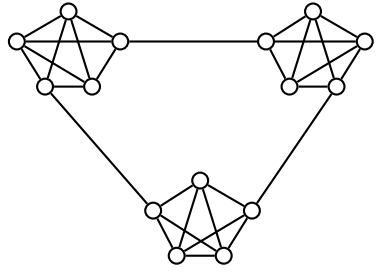

The results in the next remark follow from results stated in Section \ref{s:prod-power-a}. 
\begin{rem} Let $n\ge 2$.
\ben[$(1)$]
\item $\thpda(Q_d)=\gamma(Q_d)=2^{2^\ell-\ell-1}$ for $d=2^\ell-1$ using Corollary \ref{c:squeeze} because the order of $Q_d$ is $2^d$, $\Delta(Q_d)=d$, and for  $d=2^\ell-1$, $\gamma(Q_d)=2^{2^\ell-\ell-1}$ by \cite{HL93}.
\item $\thpda(P_r\circ K_1)=2\lc  \frac r 3\rc$ by Theorem \ref{half-ratio} and $\gamma(P_n)=\lc \frac n 3\rc$. 
\item $\thpda(C_r\circ K_1)=2\lc  \frac r 3\rc$ by Theorem \ref{half-ratio} and $\gamma(C_n)=\lc \frac n 3\rc$. 
\item $\thpda(N_{j,d})=j$ by Corollary \ref{c:squeeze} because the order of $N_{j,d}$ is $jd$, $\Delta(N_{j,d})=d-1$, and $\gamma(N_{j,d})=\lc \frac {jd} {(d-1)+1}\rc=j.$
\een
\end{rem}

We now construct  a family $G(n,s,m)$ of  $2$-connected graphs for which $\thpda(G(n,s,m))$ is less than both $\gamma(G(n,s,m))$ and $\pd(G(n,s,m)) \ppt(G(n,s,m))$. These examples lead to a family of $r$-connected graphs with the same properties, for any integer $r\geq 2$. 

 Let $K_n$ be the complete graph on  vertices $\{u_1, u_2, \ldots, u_n\}$. Replace each edge $u_iu_j$ of $K_n$ with $s\ge 1$ disjoint paths of length $m\ge 1$ between $u_i$ and $u_j$, for $1\leq i<j\leq n$. Call the resulting graph $G(n, s, m)$;    Figure~\ref{g-ns} shows $G(3,3,4)$.

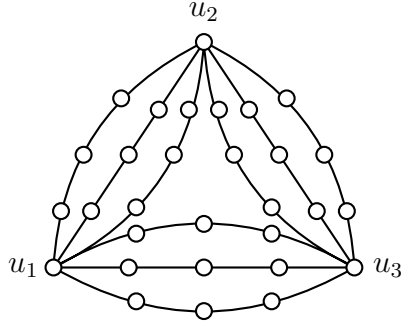
\begin{figure}[h] 
\begin{center}
\begin{tikzpicture}

\node(A) at (-.4,0) {$u_1$};
\node(B) at (2, 3.4) {$u_2$};
\node(C) at (4.45, 0) {$u_3$};

\filldraw [fill=white, thick]
(0,0) circle [radius=3pt]
(2,3) circle [radius=3pt, thick]
(4,0) circle [radius=3pt, thick]

(1,0) circle [radius=3pt]
(2,0) circle [radius=3pt]
(3,0) circle [radius=3pt]

(1.1,.45) circle [radius=3pt]
(2,.58) circle [radius=3pt]
(2.9,.45) circle [radius=3pt]

(1.1,-.45) circle [radius=3pt]
(2,-.58) circle [radius=3pt]
(2.9,-.45) circle [radius=3pt]

(.5,.75) circle [radius=3pt]
(1,1.5) circle [radius=3pt]
(1.4,2.1) circle [radius=3pt]

(1.1,.8) circle [radius=3pt]
(1.6,1.5) circle [radius=3pt]
(1.8,2.1) circle [radius=3pt]

(.1,.75) circle [radius=3pt]
(.4,1.5) circle [radius=3pt]
(.9,2.25) circle [radius=3pt]

(3.5,.75) circle [radius=3pt]
(3,1.5) circle [radius=3pt]
(2.6,2.1) circle [radius=3pt]

(2.9,.8) circle [radius=3pt]
(2.4,1.5) circle [radius=3pt]
(2.2,2.1) circle [radius=3pt]

(3.9,.75) circle [radius=3pt]
(3.6,1.5) circle [radius=3pt]
(3.1,2.25) circle [radius=3pt]
;
\begin{scope}[on background layer]
\draw[thick] (0,0)--(2,3)--(4,0)--(0,0);
\draw[thick] (0,0) to  [bend right] (2,3) to [bend right] (4,0) to [bend right] (0,0);
\draw[thick] (0,0) to  [bend left] (2,3) to [bend left] (4,0) to [bend left] (0,0);

\end{scope}

\end{tikzpicture}
\end{center}
\caption{$G(3,3,4)$ is formed by replacing each edge of  $K_3$  by $3$ paths of length $4$. }

\label{g-ns}
\end{figure}

\begin{prop}\label{p:Gnsm} Let $n\ge 2$, $s\ge 3$, and $m\ge 4$. Then 
\ben[$(1)$]
\item $\gamma(G(n,s,m))\geq s \frac{n(n-1)}{2}  \lceil \frac{m-3}{3}\rceil$. \label{gamma-gnsm}
\item  $\pd(G(n,s,m))=n-1$, $\ppt(G(n,s,m))=m$, and\\ $\pd(G(n,s,m))\cdot \ppt(G(n,s,m))=(n-1)m$. \label{pd-gnsm} 
\item  $\thpda(G(n,s,m))\le n\lc\frac{m-1}2\rc$. \label{thx-gnsm} 
\een
For $n,m\ge 5$, $\gamma(G(n,s,m))\geq \pd(G(n,s,m))\cdot \ppt(G(n,s,m))> \thpda(G(n,s,m))$. \end{prop}
\begin{proof} \eqref{gamma-gnsm}: Since the domination number of a path is the ceiling of its number of vertices divided by 3 and $m+1-4$ vertices are not accessible from the original vertices of the $K_n$, any dominating set  of $G(n,s,m)$ will have to contain at least  $\lceil \frac{m-3}{3}\rceil$ vertices from each of the $s$ paths between $u_i$ and $u_j$ for $1\leq i<j\leq n$. 

\eqref{pd-gnsm}: Note that $\{u_1, u_2, \ldots, u_{n-1}\}$ is a power dominating set of $G(n,s,m)$, because each path between vertices $u_i$ and $u_j$ will have a degree 2 neighbor of a blue endpoint turn blue in the first round.  To show that $\pd(G(n,s,m))\geq n-1$, if $S$ is a minimum power dominating set that contains neither $u_i$ nor $u_j$, then $S$ must contain at least $s-1\ge 2$ vertices on the paths between them, and replacing these  vertices by $u_i$ results in a power dominating set $S'$ with $|S'|< |S|$, contradicting the minimality of $S$.  Thus the only minimum power dominating sets are $\{u_1, u_2, \ldots, u_n\}\setminus\{u_i\}$ for $i=1,\dots,n$. To evaluate $\ppt(G(n,s,m);\{u_1,\dots,u_{n-1}\})$, note that the last vertex to be observed will be $u_n$ in round  $m$.

 \eqref{thx-gnsm}: Note that $\ppt(G(n,s,m);\{u_1,\dots,u_{n}\})=\lc \frac {m-1} 2\rc$, so $\thpda(G(n,s,m);\{u_1,\dots, u_{n}\})=n\lc \frac {m-1} 2\rc$.

Let $n,m\ge 5$.  Then $nm\ge 2m+3n$, so $n(n-1)(m-3)\ge 2(n-1)m$.  Since $s\ge 3$,  $s \frac{n(n-1)}{2}   \frac{m-3}{3}\ge (n-1)m$.  Also $nm> n+2m$ implies  $(n-1)m>\frac{n(m+1)}2=n\lp\frac{m-1}2+1\rp>n\lc\frac{m-1}2\rc$. \end{proof}

 We now construct a family of  $(r+2)$-connected graphs for any positive integer $r\geq 1$ for which 
$\thpda(G)<\min\{\gamma(G),\pd(G)\ppt(G)\}$.

 \begin{thm}  Let $r\geq 1$ and let  $s\geq 2^{r+1}+1$. Define $H_r=G(3,s,4) \Box Q_{r}$. Then 
$\thpda(H_r)<\gamma(H_r)$, $\thpda(H_r)< \pd(H_r) \ppt(H_r)$, and $H_r$ is $(r+2)$-connected.
\end{thm}
 \begin{proof} 
Clearly $H_r$ contains $2^{r}$ copies of $G(3,s,4)$. By Proposition~\ref{p:Gnsm}(1), 
$\gamma(G(3,s,4))\geq 3s\geq 3(2^{r+1}+1)$. Let $S$ be a power dominating set of $H_r$.  

Assume first that $|S|\leq 2^{r+1}-1$. Then 
there is a copy of $G(3,s,4)$ in which at most one of the vertices from the original $K_3$ are chosen. Hence there are at least two vertices in the original $K_3$ of that copy, say $u_1$ and $u_2$, that are not in $S$. For at least $s-1$ of the paths between $u_1$ and $u_2$, there must be a vertex in $S$ that is adjacent to a vertex on the path. No vertex in a different copy of $G(3,s,4)$ is adjacent to two of these paths, so there are at least $s-1=2^{r+1}$ vertices needed, contradicting our assumption that $|S|\leq 2^{r+1}-1$. Thus any power dominating set of $H_r$ must contain at least $2^{r+1}$ vertices. 

Now suppose  $S$ is a  set that has the same two vertices in each copy of $G(3,s,4)$, so $|S|=2\cdot 2^r$.  Since  power domination will occur simultaneously in each copy, $S$ is a power dominating set of $H_r$ and $\ppt(H_r;S)=\ppt(G(3,s,4);S)=\ppt(G(3,s,4))$. Thus $\pd(H_r)= 2^{r+1}$. Since any minimum power dominating set must have this form, $\ppt(H_r)=\ppt(G(3,s,4))$.  By Proposition~\ref{p:Gnsm}(2), $\ppt(G(3,s,4))=4$ and therefore $ \pd(H_r) \ppt(H_r) = 2^{r+1} \cdot 4=2^{r+3}$. 

Finally, suppose $S$ is the set of three original vertices in each copy of $G(3,s,4)$. Then $|S|=3\cdot 2^{r}$. As in the proof of Proposition~\ref{p:Gnsm}(3), the number of rounds needed to complete the power domination  simultaneously in each copy is 2. Thus, \[\thpda(H_r)\leq 3\cdot2^{r} \cdot 2 = 3\cdot 2^{r+1}
<3(2^{r+1}+1) < 4\cdot 2^{r+1}=2^{r+3}.\]

If $G$ is $s$-connected and $H$ is $r$-connected, then $G\Box H$ is $(s+r)$-connected \cite{S57}. Thus $Q_r$ is $r$-connected and $G(3,s,4)\,\Box\, Q_r$  is $(2+r)$-connected. \end{proof}

\subsection{The initial cost definition  $\thpdx(G)$}\label{s:thpdx}

In this section we summarize some basic results about $\thpdx(G)$, and we prove that $\thpdx(G)<|V(G)|$ when $G$ is a connected graph of order at least three. In contrast, we present examples where $\thpdx(G)=\frac{6}
{7}|V(G)|$.  The graphs $G=H\circ K_1$ have high domination number, but we show that $\thpdx(H\circ 
K_1)\leq \frac{3}{4}|V(G)|$ if $H$ is connected and nontrivial. We compare the results found in \cite{product-power-throt} about $\thpda(P_n\Box P_m)$ to new upper bounds on $\thpdx(P_n\Box P_m)$ that show that 
the best ways to power dominate are different in each case. 

For a graph $G$ of order $n$, recall that $\thpdx(G,k)=k(1+\ppt(G,k))$ and \[\thpdx(G)=\min_{\pd(G) \le k\le n}k(1+\ppt(G,k))=\min_{\pd(G) \le k \le n}\thpdx(G,k).\]

The next result follows from Observations \ref{o:bd-x} and  \ref{o:ub-x}. 

 \begin{obs}\label{o:thpd-x} For every  graph $G$  of order $n$:
\ben[$(1)$]
\item $\thpdx(G)\le \pd(G)(1+\ppt(G))$.  
\item $\thpdx(G)\le 2\gamma(G)$.
\item $\pd(G)\le  \thpdx(G)\le n$ and $\pd(G)+1\le  \thpdx(G)$ if $G$ is connected and $n\ge 2$.
\item If $\pd(G)\ge \frac n 2$, then $\thpdx(G)=n$.
\een
\end{obs}

 The next result follows from Remark~\ref{r:univ-low} in the case $G$ is connected, and the analysis of the disconnected case is straightforward.

\begin{rem} $\null$ \label{low-thx} 
\ben[$(1)$]
\item  $\thpdx(G)=1$  if and only if $G=K_1$.
\item  $\thpdx(G)=2$ if and only if $\gamma(G)=1$ or $G=2K_1$.  
\item  $\thpdx(G)=3$ if and only if  $\ppt(G,1)=2$ or $G=3K_1$ or $G=K_2\du K_1$.
\een
 A description of a construction for a connected graph $G$ with $\ppt(G,1)=2$, which is equivalent to $\pd(G)=1$ and $\ppt(G)=2$, appears  in \cite{product-power-throt}. In particular, $\thpdx(C_4)=3$. 
\end{rem}

\begin{rem}  \label{thx-families} Power domination on paths behaves like Cops and Robbers, and power domination on $C_n$ behaves like power domination on $P_n$.  The remaining parts of the next result follow from  Observation \ref{o:thpd-x} and Remark \ref{low-thx}.
\ben[$(1)$]
\item $\thpdx(P_n)=1+\rad(P_n)=1+\lc\frac{n-1}2\rc$.
\item $\thpdx(C_n)=1+\rad(P_n)=1+\lc\frac{n-1}2\rc$.
\item  $\thpdx(K_n)=2$.
\item $\thpdx(K_{1,n-1})=2$.
\item For $n\ge 4$, $\thpdx(K_{2,n-2})=3$.
\item For $p,q\ge 3$, $\thpdx(K_{p,q})=4$.
\een
\end{rem}

In Theorem \ref{less-than-n} we show that $\thpdx(G)<|V(G)|$ for a connected graph of order at least three, using the next result.

\begin{thm}\label{dom-book-2}{\rm \cite[Theorem 2.2]{dom-book}} A  connected graph  $G$  of order $n\ge 2$ has $\gamma(G)=\frac n 2$ if and only if $G=H\circ K_1$ for some connected graph $H$ or  $G=C_4$.  \end{thm}

\begin{thm}\label{less-than-n} Let   $G$ be a connected graph of order $n\ge 3$.  Then $\thpdx(G)<n$. Furthermore, if $G=H\circ K_1$ for a connected graph $H$ of order at least two, then $\thpdx(G)=3\gamma(H) $ and $\thpdx(G)= \frac{3n}{4}$.
\end{thm}

\begin{proof} 
By Theorem~\ref{half}, $\gamma(G) \leq \frac{n}{2}$. If $\gamma(G)< \frac{n}{2}$, then by Observation \ref{o:thpd-x}(2), $\thpdx(G)\leq 2\gamma(G)<n$.
So assume $\gamma(G)=\frac{n}{2}$. By Theorem~\ref{dom-book-2}, either $G=C_4$ or $G$ consists of a connected graph $H$ with a leaf attached to each vertex.
If $G=C_4$, then $\thpdx(C_4)=3<n$ by Remark~\ref{low-thx}(3).

So suppose $G=H\circ K_1$ for a connected graph $H$. 
Since $n\ge 3$, $H$ has at least 2 vertices, so by Theorem~\ref{half}, $\gamma(H)\leq \frac{|V(H)|}{2}$, and this proves  that $3\gamma(H)\leq \frac{3n}{4}$. Let $S$ be a dominating set of $H$ with $|S|=\gamma(H)$. Then $S$ is a power dominating set of $G$ with propagation time two.  Thus  $\thpdx(G)\leq 3\gamma(H) \leq \frac{3n}{4}<n$.

To show that $\thpdx(G)=3\gamma(H)$ for  $G=H\circ K_1$, we choose $S\subset V(G)$ such that $\thpdx(G;S)=\thpdx(G)$ and show that $\thpdx(G;S)\ge 3\gamma(H)$. Note that $S$ is not a dominating set of $G$  since $\gamma(G)=\frac n 2$ and $\thpdx(G)<n$.  
Without loss of generality, we may assume that $S\subseteq V(H)$ since we can always replace a leaf of $G$ by its neighbor in $H$. The set $S$ must be a dominating set of $H$ in order to be a power dominating set of $G$, so $|S|\geq \gamma(H)$ and the propagation time for $S$ is two. Therefore $ \thpdx(G;S)=3|S|\ge  3\gamma(H)$. 
\end{proof}

The next result is immediate from Theorem \ref{less-than-n}.
\begin{cor}\label{c:Pm-circ-K1} For $m\ge 2$, $\thpdx(P_m\circ K_1)= 3\lc \frac{m}{3}\rc$.
\end{cor}

Next we present a family of connected graphs that satisfy $\thpdx(G)=\frac{6|V(G)|}{7}$. 

\begin{ex} \label{ex-leslie} Let $G_1$ consist of a $P_4$ and a $P_3$ with an edge connecting them between a vertex of degree 2 on each, as shown in Figure~\ref{leslies}. The vertices of degree 3 are colored green in the figure. Then $\gamma(G_1)=3$, hence $\thpdx(G_1)\leq 3\cdot 2 =6$. However, any power dominating set of $G_1$ that is not a dominating set must contain a vertex from each path and have propagation time at least 2. Thus $\thpdx(G_1)\geq 2\cdot 3=6$. Since $|V(G_1)|=7$, $\thpdx(G_1)=\frac{6|V(G_1)|}{7}$. 

Now let $G$ consist of  $r$ disjoint copies of $G_1$, say $G_1, G_2, \ldots, G_r$, with any subset of edges between the green vertices in each copy of $G_1$. Any dominating set of $G$ must contain at least 3 vertices from each $G_i$, $1\leq i\leq r$, in order to dominate the leaves. Any power dominating set of $G$ must contain at least two vertices from each $G_i$, $1\leq i\leq r$. Hence $\thpdx(G)=\frac{6|V(G)|}{7}$. \end{ex} 

\begin{figure}[h] 
\begin{center}
\begin{tikzpicture}
\filldraw [fill=white, thick]
(0,1) circle [radius=3pt]
(2,1) circle [radius=3pt]
(3,1) circle [radius=3pt]
(0,0) circle [radius=3pt]
(2,0) circle [radius=3pt];
\filldraw [fill=green, thick]
(1,1) circle [radius=3pt]
(1,0) circle [radius=3pt];

\draw[thick] (0,1)--(3,1);
\draw[thick] (1,0)--(1,1);
\draw[thick] (0,0)--(2,0);
\node at (1.5,-.8) {$G_1$};

\filldraw [fill=white, thick]
(4,1) circle [radius=3pt]
(6,1) circle [radius=3pt]
(7,1) circle [radius=3pt]
(4,0) circle [radius=3pt]
(6,0) circle [radius=3pt];
\filldraw [fill=green, thick]
(5,1) circle [radius=3pt]
(5,0) circle [radius=3pt];

\draw[thick] (4,1)--(7,1);
\draw[thick] (5,0)--(5,1);
\draw[thick] (4,0)--(6,0);
\node at (5.5,-.8) {$G_2$};

\draw[green] (1,1) to [bend left] (5,1);

\node at (8,.5){$\ldots$};
\node at (9,.5){$\ldots$};

\filldraw [fill=white, thick]
(10,1) circle [radius=3pt]
(12,1) circle [radius=3pt]
(13,1) circle [radius=3pt]
(10,0) circle [radius=3pt]
(12,0) circle [radius=3pt];
\filldraw [fill=green, thick]
(11,1) circle [radius=3pt]
(11,0) circle [radius=3pt];

\draw[thick] (10,1)--(13,1);
\draw[thick] (11,0)--(11,1);
\draw[thick] (10,0)--(12,0);
\node at (11.5,-.8) {$G_r$};
\draw[green] (5,1) to [bend left] (11,1);
\draw[green] (1,0) to [bend right] (5,0);

\end{tikzpicture}\vspace{-5pt}
\end{center}

\caption{A graph $G$ with $|V(G)|=7r$ and $\thpdx(G)=\frac{6|V(G)|}{7}$. Any subset of edges between the green vertices may be included.}

\label{leslies} 
\end{figure}
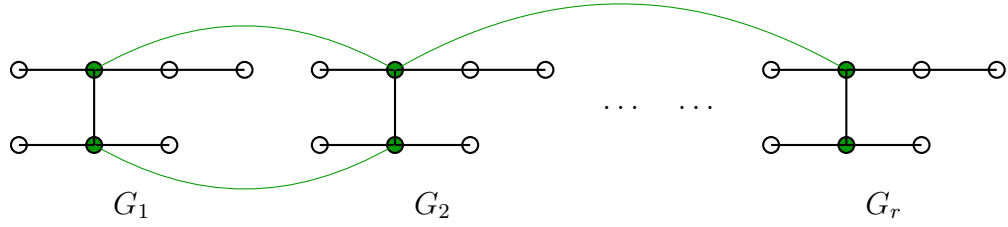

Let $G$ be  $j$ copies of $P_3$ and $i$ copies of $P_4$ (all disjoint, so the order of $G$ is $ 3j+4i$). We show $\frac{\thpdx(G)}{3j+4i}\geq \frac{6}{7}$ if and only if $j=i$.  Note that $\gamma(G) = j+2i$,  $\pd(G)=j+i$, and $\ppt(G)=2$. Then $\thpdx(G) = \min\{2(j+2i), 3(j+i)\}$ because $\ppt(G,k)=2$ for $\pd(G)\le k<\gamma(G)$. Thus $i=j$ implies $\thpdx(G)=\frac 6 7$.  By choosing a dominating set,
$\frac{\thpdx(G)}{3j+4i} \leq \frac{2\gamma(G)}{3j+4i}=\frac{2(j+2i)}{3j+4k}=\frac{2j+4i}{3j+4i}$. If $i<j$, then
 $14j+28i< 18j+24i$, which implies $\frac{2j+4i}{3j+4i}<\frac{6}{7} $. By choosing  a minimum  power 
dominating set, $\frac{\thpdx(G)}{3j+4i} \leq \frac{3(j+i)}{3j+4i}=\frac{3j+3i}{3j
+4i}.$ If $j<i$, then $21j+21i<18j+24i$, which implies $ \frac{3j+3i}{3j
+4i}<\frac 6 7.$ Hence $i\ne j$ implies $\frac{\thpdx(G)}{3j+4i}< \frac{6}{7}$.

 If $G$ is connected and not $K_1$ or $K_2$, then Theorem~\ref{less-than-n} shows that the ratio $\frac{\thpdx(G)}{|V(G)|}$ is less than 1, and Example~\ref{ex-leslie} shows that it can be as large as  $\frac{6}{7}$.
\begin{quest} Is $\frac 6 7$ the largest possible value of $\frac{\thpdx(G)}{|V(G)|}$ that is achieved for connected graphs of arbitrarily large order?
\end{quest}

Next we examine grid graphs, which are natural to consider in PMU placement problems.  It is interesting to compare the value of  $\thpdx(P_m\circ K_1)$ in Corollary \ref{c:Pm-circ-K1} with the value of  $\thpdx(P_m\Box P_2)$ in Theorem \ref{th-p2-pm}.

\begin{prop} \label{grid-min} For $n,m\ge 2$, $\thpdx(P_n\Box P_m)\leq \min\{\lc \frac{m}{3}\rc(n+1), \lc \frac{n}{3}\rc(m+1)\}$.
\end{prop}

\begin{proof} Arrange $P_n\Box P_m$ with $n$ rows and $m$ columns. 
By symmetry, we need only show that $\thpdx(P_n\Box P_m)\leq \lc \frac{m}{3}\rc(n+1)$. Let $S$ be a minimum dominating set of the top row of $P_n\Box P_m$, so $|S|=\gamma(P_m)$. After the first round, each vertex in the top row is observed and has at most one unobserved neighbor in the second row.  Thus zero forcing can proceed row by row, so $\ppt(P_n\Box P_m;S)\le n$ and $\thpdx(P_n\Box P_m)\leq \lc \frac{m}{3}\rc(1+n)$. 
\end{proof}

In \cite{product-power-throt}, it was shown that  $\thpda(P_n\Box P_m)= \gamma(P_n\Box P_m)$ for all $m,n$ (see Theorem \ref{t:grid}). As for all graphs, $\thpdx(P_n\Box P_m) \leq 2\gamma(P_n\Box P_m)$. For example, $\gamma(P_2\Box P_m) = \lf \frac{m+2}{2}\rf $ and $\gamma(P_3\Box P_m)=\lf\frac{3m+4}{4}\rf$ (see \cite{ACIOP11grid}), and therefore 
$\thpdx(P_2\Box P_m) \leq 2\gamma(P_2\Box P_m) = 2\lf \frac{m+2}{2}\rf\le m+2$ and
$\thpdx(P_3\Box P_m) \leq 2\gamma(P_3\Box P_m) = 2\lf\frac{3m+4}{4}\rf $. However, these bounds are not tight, as shown in the next remark.

\begin{rem} In \cite{GoPiRaTh2011}, it is shown that $\gamma(P_n\Box P_m) = \lf \frac{(m+2)(n+2)}{5}-4\rf$ for $m,n\geq 16$. Since $\lc\frac{m}{3}\rc(n+1)<2\left(\lf\frac{(m+2)(n+2)}{5}\rf-4\right)$, $\thpdx(P_n\Box P_m)< 2\gamma(P_n\Box P_m)$  for $m,n\geq 16$. 
\end{rem} 

  To establish the exact value of $\thpdx(P_2\Box P_m)$ in Theorem \ref{th-p2-pm}, we need some definitions. We use  the notation $u\to v$ or {\em $u$ forces $v$}  to mean $u$ observes $v$ (this may involve a choice among several vertices that can observe $v$). For  a power dominating set  $S\subseteq V(G)$, 
create a \emph{propagating power domination set of forces} $\F$ of $S$ as follows: Initially, $\F=\emptyset$.
 For each $w\in N[S]\setminus S$, choose $x\in S\cap N(w)$ and add $x\to w$ to $\F$.  Then choose a propagating set of forces for $N[S]$ (using the standard color change rule for zero forcing) and add that to $\F$.
 Suppose  $S$ is a power dominating set of $G$ and $\F$ is a propagating power domination set of forces of $S$. For a vertex $x\in S$, define $S_x$ to be the set of all vertices  $w$ such that  there is a sequence of forces  $x=v_0\to v_1\to\dots\to v_k=w$ in $\F$; the empty sequence of forces is permitted, i.e., $x\in S_x$.

\begin{thm} \label{th-p2-pm} For $m\ge 2$, $\thpdx(P_2\Box P_m)= m$ if $m\equiv 0\! \mod 3$  and $\thpdx(P_2\Box P_m) = m+1$ if $m\not\equiv 0\! \mod 3$. \end{thm}
\begin{proof} Let $G=P_2\Box P_m$.
By Proposition~\ref{grid-min}, $\thpdx(G) \leq \lc \frac{2}{3}\rc(m+1)=m+1$.  If $m$ is divisible by 3, then $\thpdx(G) \leq3\lc\frac{m}{3}\rc= 3\cdot \frac{m}{3}=m$. This proves the upper bound. Next we prove the lower bound.

For any $k$ such that $\ppt(G,k)=1$,  $k\ge \gamma(G)= \lf \frac{m+2}{2}\rf $ by \cite{ACIOP11grid}, so $\thpdx (G,k)\ge 2 \lf \frac{m+2}{2}\rf\ge m+1$. Thus we need consider only sets $S$ such that  $\ppt(G;S)\geq 2$.

Arrange $G$ with 2 rows and $m$ columns. 
Choose $S$ such that $\thpdx(G)=\thpdx(G;S)=|S|(1+p)$ where $p=\ppt(G;S)$. Create a propagating power domination set of forces  for $S$ by choosing the forcing vertex in the same row whenever there is a choice (row-forcing is preferred).   For $x\in S$, we show that $|S_x|\le 2(1+p)$ and $|S_x|< 2(1+p)$ except under the additional conditions that $|P^{(1)}(S)\cap S_x|= 3$, $|P^{(2)}(S)\cap S_x|= 2$, and $p=2$.

By the rules of power domination, $|P^{(i)}(S)\cap S_x|\le |P^{(1)}(S)\cap S_x|$ for $i\ge 2$.
Thus $|P^{(1)}(S)\cap S_x|\le 2$ implies  $|S_x| \le 1+2p<2(1+p)$. So assume $|P^{(1)}(S)\cap S_x|= 3$ and $x$ is in the top row.   Denote the east, south, and west neighbors of $x$ by $x_E, x_S$ and $x_W$, and name additional vertices similarly, according to their direction from $x$. Since row-forcing is preferred, $x_{SE}, x_{SW}\not\in S$.

Let $r\ge 2$ be the first round in which any of $x_W, x_S, x_E$ performs a force. We analyze the situation based on which force(s) occur in round $r$.
To obtain a contradiction, suppose that $x_W\to x_{WW}$  in round $r$, which requires $x_{SW}\in P^{[r-1]}(S)$  and $x_{SW}$ is not forced by $x_S$. If  $x_{SW}\in P^{[1]}(S)$, then   $x_{WSW}\in S$ and so $x_{WW}\in P^{(1)}(S)$, which is a contradiction.
Otherwise, $x_{SW}\in P^{(i)}(S)$ with $1<i<r$ and $x_{WSW}\to x_{SW}$ in round $i$ requires that $x_{WW}$ must already be observed before round $i$, which is a contradiction. Therefore, $x_W\to x_{WW}$ cannot happen in round $r$. Similarly, $x_E\to x_{EE}$ cannot happen in round $r$.

Now suppose that $x_S\to x_{SW}$  in round $r$.  This requires $x_{SE}\in P^{[r-1]}(S)$, which in turn requires $x_{ESE},x_{EE}\in P^{[r-1]}(S)$. So  $x_E$ can never force.
Then  $|P^{(r)}(S)\cap S_x|=1$ since $x_{W}$ cannot force in round $r$,  $|P^{(i)}(S)\cap S_x|\le 2$ for $i\ge r+1$, and  $|P^{(i)}(S)\cap S_x|=0$ for $1<i<r$.  Thus \vspace{-8pt}
\[|S_x|\le \sum_{i=0}^p |P^{(i)}(S)\cap S_x|\le 1+3+0+\dots+0+1+2(p-r)<2(1+p).\vspace{-4pt}\]
The case $x_S\to x_{SE}$  in round $r$ is similar. 

It remains to consider the case when $x_S$ does not force in round $r$, and in this case, $|P^{(r)}(S)\cap S_x|\le 2$.
By definition of $r$, one of $x_W, x_E$ must force in round $r$, and we have shown that it cannot force along a row.  Without loss of generality, let $x_W\to x_{SW}$ in round $r$.  Necessarily, $x_{WW}\in P^{[r-1]}(S)$ or $x_{W}$ and $x_{SW}$ are the leftmost vertices in $G$. 
We show that  $x_{WSW}\in P^{[r]}(S)$ if $x_{WW}\in P^{[r-1]}(S)$, and so in either case $x_{SW}$ never performs a force. 
If $x_{WW}\in S$, then $x_{WSW}\in P^{[1]}(S)\subset P^{[r]}(S)$. 
If $x_{WWW}\to x_{WW}$  in round $r-1$, then $x_{WW}\to x_{WSW}$  in round $r$.   
Note that $x_{WSW}$ cannot force $x_{WW}$ in round $r-1$ because its neighbor $x_{SW}$ is also unobserved until round $r$.
If $x_E\to x_{SE}$ in round $r$, a similar proof shows that $x_{SE}$ can never force. Thus $|P^{(i)}(S)\cap S_x|\le 2$ for $i\ge r+1$, and  \vspace{-8pt}
\[|S_x|\le \sum_{i=0}^p |P^{(i)}(S)\cap S_x|\le 1+3+0+\dots+0+2+2(p-r)\le 2(1+p).\vspace{-4pt}\]
If $r\geq 3$, then $|P^{(2)}(S)\cap S_x|=0$, and $|S_x|<2(1+p)$.
If $r=2$ and {$x_W\not\to x_{SW}$ or $x_E\not\to x_{SE}$}, then {$x_{W}$ or $x_{E}$} cannot force in round 2, and $|S_x|\le 1+3+1+2(p-2)<2(1+p)$.
If $r=2$ and {$x_W\to x_{SW}$ and $x_E\to x_{SE}$}  in round $2$, then as described above, {none of $x_{SW}$, $x_{SE}$, or $x_{S}$} can ever force,   so $|P^{(i)}(S)\cap S_x|=0$ for $i\ne 0,1,2$ and $ |S_x|< 2(1+p)$ if $p>2$. Thus we have shown that  $|S_x|\leq 2(1+p)$ in all cases, and $|S_x|< 2(1+p)$
unless  $|P^{(1)}(S)\cap S_x|= 3$, $|P^{(2)}(S)\cap S_x|= 2$, and $p=2$.

Since $|S_x|\le 2(1+p)$ in all cases, $2m=\sum_{x\in S}|S_x|\le |S|2(1+p)=2\thpdx(G)$ and  
 $\thpdx(G)\ge m$ for all $G$.  Furthermore $\thpdx(G)>m$  unless $|P^{(1)}(S)\cap S_x|= 3$, $|P^{(2)}(S)\cap S_x|= 2$, and $p=2$ for all $x\in S$. In this case $|S_x|=6= 2(1+p)$ for all $x\in S$ and hence $2m=\sum_{x\in S}|S_x|= 6|S|$ and $m=3|S|$, and $m$ is divisible by 3. \end{proof}


\section{Product throttling for PSD zero forcing}\label{s:prod-Z-PSD}

The next result implies that product throttling for $\Zp$  is  nontrivial for $\thpx(G)$ and $\thpa(G)$, based on results from Cops and Robbers. As in Section \ref{s:prod-univ}, define $\kp(G,p)=\min\{|S|:\ptp(G;S)=p\}$.

\begin{thm}\label{t:ptc-le-ptp}{\rm \cite{cop-throttle}}  Let $S\subseteq V(G)$ be a PSD  zero forcing set.  Then $|S|\ge c(G)$, $\capt(G;S)\le \ptp(G;S)$,    
and $c(G)\le \Zp(G)$.  
If $T$ is a tree and $S\subseteq V(T)$, then $\capt(T;S)=\ptp(T;S)=\ecc(S)$ for $S\subseteq V(T)$.  Thus $\capt_k(T)=\ptp(T,k)=\rad_k(T)$.
\end{thm}

Note that $c(G)$ and $\Zp(G)$ can be substantially different, resulting in very different product throttling numbers.  For example, $\Zp(K_n)=n-1$ whereas $c(K_n)=1$.

\subsection{Initial cost definition $\thpx(G)$}

Let $G$ be a graph of order $n$. Define $\thpx(G,k)=k(1+\ptp(G,k))$ and \[\thpx(G)=\min_{\Zp(G) \le k \le n}\thpdx(G,k)=\min_{\Zp(G) \le k\le n}k(1+\ptp(G,k)).\]

The  results in the next remark follow immediately from the universal forms of these results in Section \ref{ss:univ}.

   \begin{rem}\label{o:thpx-from-univ} Let $G$ be a  graph of order $n$.
  \ben[$(1)$]
   \item $\Zp(G)\le  \thpx(G)\le n$.  
 If $G$ is connected and $n\ge 2$, then $\Zp(G)+1\le  \thpx(G)$. 
 \item $\thpx(G)\le \Zp(G)(1+\ptp(G))$.  
\item $\thpx(G)\le 2k_+(G,1)$.
\item $\thpx(G)\ge  \min_{\Zp(G)\le k\le n}k(1+\rad_k(G))$.
\item If $\Zp(G)\ge \frac n 2$, then $\thpx(G)=n$. Examples include $\thpx(K_n)=n$ and $\thpx(Q_d)=2^d$.   \een
  \end{rem}

The next result follows from Theorems \ref{t:thcx-chord} and \ref{t:ptc-le-ptp} since a tree is chordal.

\begin{cor} Let $T$ be a tree.  Then  $ \thpx(T,k)= \thcx(T,k)$ and $ \thpx(T)= \thcx(T)=1+\rad T$.  
In particular, $\thpx(P_n)=1+\lc\frac{n-1}2\rc$. 

For any graph $G$,  $ \thcx(G,k)\le \thpx(G,k)$ 
for $c(G)\le k\le n$. 
\end{cor}

  The next result follows from Remark~\ref{r:univ-low} in the case $G$ is connected, and the analysis of the disconnected case is straightforward (for $n=4$ see \cite{sage}).

\begin{rem} Let $G$ be a graph. 
\ben[$(1)$]
\item $\thpx(G)=1$  if and only if $G=K_1$.
\item $\thpx(G)=2$ if and only if $\Zp(G)=1$ and $\ptp(G)=1$ ( i.e., $G=K_{1,n-1}$) or $G=2K_1$.  
\item\label{px3} $\thpx(G)=3$ if and only if $G$ satisfies exactly 
one of the following conditions:  
		\ben[(a)]
			\item $G=K_3$, $G=K_2\du K_1$, or $G=3K_1$.
			\item $\Zp(G)=1$  and    $\ptp(G,1)=2$ (i.e., $G$ is a tree and $\rad G=2$). 
		\een
\item $\thpx(G)=4$ if and only if  $G$ satisfies at least one of the following conditions:
		\ben[(a)]
			\item $G=K_4$, $G=K_3\du K_1$, $G=K_2\du 2K_1$, or $G=4K_1$.
			\item $\ptp(G,2)=1$ and $\ptp(G,1)> 3$.
			\item $\Zp(G)=1$ and $\ptp(G,1)=3$.
		\een
\een
\end{rem}

\subsection{No initial cost  definition $\thpa(G)$}

Let $G$ be a connected graph of order $n\ge 2$. Define $\thpa(G,k)=k\ptp(G,k)$ and \[\thpa(G)=\min_{\Zp(G)\le k< n}k\ptp(G,k)=\min_{\Zp(G)\le k < n}\thpa(G,k).\]

The  results in the next remark follow immediately from the universal forms of these results in Section \ref{ss:univ}.

   \begin{rem}\label{o:thpa-from-univ} Let $G$ be a connected graph of order $n\ge 2$.
   \ben[$(1)$]
   \item $\Zp(G)\le  \thpa(G)\le \kp(G,1)\le n-1$.  
   \item $\thpa(G)\le \Zp(G)\ptp(G)$ and $ \thpa(G)=\Zp(G)$ if and only if $\ptp(G)=1$.
\item If $G'$ is a connected graph  of order $n'\ge 2$ that is an induced subgraph of  $G$,  and $\thpa(G')=\kp(g,1)$, then
 \[\thpa(G)\le  n-n'+\thpa(G').\]
 \item $\thpa(G)\ge \thp(G)-1$.  
 \item If $\thp(G)=\thp(G,1)=1+\ptp(G,1)$ or $\thp(G)=\thpa(G,\kp(G,1))=\kp(G,1)+1$, then $\thpa(G)=\thp(G)-1$.
 \item $\thpa(G)\ge  \min_{\Zp(G)\le k< n}k\rad_k(G)$
 \een
  \end{rem}

The next result follows from Theorem 
\ref{t:ptc-le-ptp}. 

\begin{cor} Let $T$ be a tree.  Then   $ \thpa(T,k)= \thca(T,k)$ and \[ \thpa(T)= \thca(T)=\min_{\Zp(G)\le k< n}k\rad_k(G).\]  
For any graph $G$,   $ \thca(G,k)\le \thpa(G,k)$ for $c(G)\le k< n$. 
\end{cor}

The next remark lists values of $\thpa(G)$ for various  families of graphs $G$. 

\begin{rem}\label{p:families}  PSD propagation time agrees with capture time for paths and cycles, so the values for these graphs follow from Remark \ref{thca:vals}.   
 For the remaining graphs $G$, it  is well known that $\ptp(G)=1$ \cite{PSDproptime}. 
\ben[$(1)$]
\item For $n\ge 2$, $\thpa(P_n)=\gamma(P_n)=\lc\frac n 3\rc$.  
\item For $n\ge 4$, $\thpa(C_n)=\gamma(C_n)=\lc\frac n 3\rc$.
\item For $n\ge 2$, $\thpa(K_n)=\Zp(K_n)=n-1$.
\item For $1\le p\le q$, $\thpa(K_{p,q})=\Zp(K_{p,q})=p$.
\item For $d\ge 1$, $\thpa(Q_d)=\Zp(Q_d)=2^{d-1}$ where $Q_d$ is the hypercube. 
\een
\end{rem}

Let $\alpha(G)$ denote the independence number of $G$.  The next result is analogous to Proposition  2.9 in \cite{PSDthrottle}.

\begin{prop}\label{alpha} $\thpa(G)\le n-\alpha(G)$.
\end{prop}
\bpf
Let $S'$ be a maximum independent set of vertices (so  $|S'|=\alpha(G)$) and let $S=V(G)\setminus S'$.  Then $\ptp(G;S)=1$ and  $\thpa(G)\le n-\alpha(G)$.
\epf\vspace{-10pt}

\subsection*{Extreme values of $\thpa(G)$}

\begin{rem} Let $G$ be a connected graph of order $n\ge 2$. From Remark \ref{r:univ-low}  we have: 

\ben[$(1)$]
\item $\thpa(G)=1$ if and only if $\Zp(G)=1$ and $\ptp(G)=1$, i.e., $G=K_{1,n-1}$.
\item $\thpa(G)=2$ if and only if 
$G$ satisfies one of the following conditions:
		\ben[(a)]
			\item $\Zp(G)\le 2$, $\ptp(G,2)=1$, and $\ptp(G,1)> 2$. 
			\item $\Zp(G)=1$  and $\ptp(G,1)=2$ (i.e., $G$ is a tree and $\rad(G)=2$).
		\een
\item 
$\thpa(G)=3$ if and only if
$G$ satisfies  
one of the following conditions:
		\ben[(a)]
			\item $\Zp(G)\le 3$, $\ptp(G,3)=1$, $\ptp(G,2)\ge 2$, and $\ptp(G,1)> 3$.
			\item $\Zp(G)=1$, $\ptp(G,1)=3$, and $\ptp(G,2)\ge 2$  (so $G$ is a tree, $\rad(G)=3$, and no two vertices dominate $G$).
		\een
\een
\end{rem}

\begin{rem}\label{th+n-1}
Since we require $G$ to be connected and of order at least two, $\thpa(G)\le n-1$.  Furthermore, $\thpa(G)=n-1$ if and only if $G=K_n$ because $\thpa(G)=n-1$ implies  $\alpha(G)=1$ by Proposition \ref{alpha}, so $G$ is complete (and   $\thpa(K_n)=n-1$). 
\end{rem}

Define the set $\HH$ to be the set of all connected graphs $G$ of order at least two such that 
 $\alpha(G)=2$ and
 $G$ does not have an induced 5-cycle or house subgraph (see  Figure \ref{fig:hithpaforbid}).
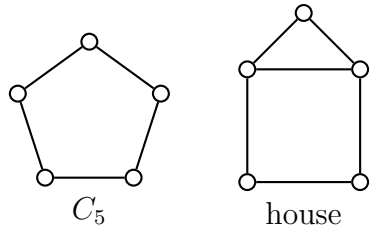
\begin{figure}[!h]
\begin{center}
\begin{tikzpicture}[scale=1]
  \vertex (v1) at (.951,.309){};
  \vertex (v2) at (0,1){};
  \vertex (v3) at (-.951,.309){};
  \vertex (v4) at (-.588,-.809){};
  \vertex (v5) at (.588,-.809){};
  \vertex[color=white,label={below:$C_5$}] (label) at (0,-.809){};
  \foreach \m/\n in {1/2,2/3,3/4,4/5,5/1}{ \draw[thick] (v\m) to (v\n); }
\end{tikzpicture}
\hspace{.25in}
\begin{tikzpicture}[scale=1]
  \vertex (v1) at (0,0){};
  \vertex (v2) at (1.5,0){};
  \vertex (v3) at (1.5,1.5){};
  \vertex (v4) at (0,1.5){};
  \vertex (v5) at (.75,2.25){};
  \vertex[color=white,label={below:house}] (label) at (.75,0){};
\foreach \m/\n in {1/2,2/3,3/4,4/1,5/3,5/4}{ \draw[thick] (v\m) to (v\n); }
\end{tikzpicture}
\caption{Graphs forbidden as induced subgraphs of $G\in\HH$.\label{fig:hithpaforbid}\vspace{-20pt}}
\end{center}
\end{figure}

\begin{rem}\label{r:G-S_comp} Suppose $G\in \HH$ and let $S\subset V(G)$.  Since $\alpha(G)=2$,  $G-S$ is connected or consists of two connected components. 
 If $G-S$ has two connected components, then each is a clique.
\end{rem}

It is straightforward to verify the statements in the next lemma. 
\begin{lem}\label{l:kp-alg} Let $n\ge 3$.  
\ben[$(1)$]
\item For $k=1,\dots,n-2$, $k(n-k-1)\ge n-2$.
\item For $k=2,\dots, n-2$, $k\frac {n-k}2\ge n-2$.
\een
\end{lem}

\begin{lem}\label{l:G-S_sum}{\rm \cite{PSDthrottle}}
Let $G\in \HH$ and let $S\subset V(G)$ be  such that  $G-S$ is  connected and $|S|\le n-3$.  Then $\ptp(G;S)=n-|S|-1$.
\end{lem}

\begin{lem}\label{l:G-S_lem1}
Let $G\in \HH$ and let $S\subset V(G)$ be  such that  $G-S$ is consists of two nonempty connected components,  $G[W_1]$ and $G[W_2]$.  Then $|S^{(j)}\cap W_i|\le 1$ for $i=1,2$ and $j=1,\dots,\ptp(G;S)$.
\end{lem}
\bpf We show that $|S^{(j)}\cap W_i|\ge 2$ and $\alpha(G)=2$ implies $G$ has an induced house graph or  5-cycle; the construction is illustrated in Figure \ref{f:G-S_lem1}.  Suppose that $w,w'\in S^{(j)}\cap W_1$ and $w\ne w'$.  Then there exist $x,x'\in S^{(j-1)}$ such that $x\to w$ and $x'\to w'$ in round $j$. This means that $xw, x'w'\in E(G)$ and $xw', x'w\not \in E(G)$.  Since there are no edges between $W_1$ and $W_2 $ and $G[W_1]$ is a clique, $x,x'\in S$ and $ww'\in E(G)$.  Let $u\in W_2$.  Then $wu,w'u\not\in E(G)$. Since $\alpha(G)=2$, $wu,wx'\not\in E(G)$ implies $ux'\in E(G)$ and $w'u,w'x\not\in E(G)$ implies $ux\in E(G)$.  Thus $G[\{u,x,w,w', x'\}]$ is a 5-cycle or house graph depending on whether $xx'\not\in E(G)$ or $xx'\in E(G)$.  
\vspace{-10pt}\epf
\begin{figure}[!h]
\centering
\begin{tikzpicture}[thick, every fit/.style={ellipse,draw,inner sep=-2pt,text width=1cm}]
\node (a1) [color=white] at (0,.5){};
\node (a2) [color=white] at (0,-1.5){};
\vertex (w) [label=above: $w$] at (0,0){};
\vertex (wprime) [label=below: $w'$] at (0,-1){};
\node [fit=(a1) (a2),label=below:$W_1$] {};
\node (b1) [color=white] at (2,.5){};
\node (b2) [color=white] at (2,-1.5){};
\vertex (x) [label=above: $x$] at (2,.2){};
\vertex (xprime) [label=below: $x'$] at (2,-1.2){};
\node [fit=(b1) (b2),label=below:$S$] {};
\node (c1) [color=white] at (4,.5){};
\node (c2) [color=white] at (4,-1.5){};
\vertex (u) [label=below: $u$] at (4,-.5){};
\node [fit=(c1) (c2),label=below:$W_2$] {};
\draw (w) -- (wprime);
\draw (w) -- (x);
\draw (wprime) -- (xprime);
\draw (x) -- (u);
\draw (xprime) -- (u);
\draw[dashed] (x) -- (xprime);
\end{tikzpicture}
\caption{Diagram for Lemma \ref{l:G-S_lem1}. The dotted edge may or may not be present. \label{f:G-S_lem1}\vspace{-10pt}}
\end{figure}
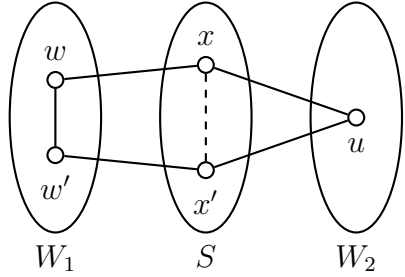

\begin{thm}  Let $G$ be a connected graph of order $n\ge 2$.  Then $\thpa(G)=n-2$ if and only if $G\in\HH$.
\end{thm}

\bpf Let $H$ denote the house graph.  Observe that $\thpa(C_5)=2$, $\ptp(C_5,2)=1$, $\thpa(H)=2$, and $\ptp(H,2)=1$.
If $G$ has an induced $C_5$ or $H$, then $\thpa(G)\le n-5+2=n-3$ by Remark \ref{r:thxa-induced} applied to PSD throttling.  If $\alpha(G)\ge 3$, then $\thpa(G)\le n-3$ by Proposition \ref{alpha}, and $\alpha(G)=1$ implies $G=K_n$ and $\thpa(G)=n-1$.  Thus $G\not\in\HH$ implies $\thpa(G) \ne n-2$.

Now assume $G\in\HH$ and let $S\subset V(G)$ such that $\thpa(G)=\thpa(G;S)=\thpa(G,k)$ for $k=|S|$.  Since $G$ is connected and  $\alpha(G)=2$, $n\ge 3$, $\thpa(G)\le n-2$, and it suffices to show that $\thpa(G,k)\ge n-2$.  Observe first that $\thpa(G,k)\ge n-2$ for $k=n-2,n-1$, so assume $k\le n-3$.  If  $G-S$ is connected, then  $\ptp(G;S)=n-k-1$ by Lemma \ref{l:G-S_sum}, so  $\thpa(G,k)\ge k(n-k-1)\ge n-2$ by Lemma \ref{l:kp-alg}.   Suppose $G-S$ is not connected.  Then $G-S$ has only two connected components.  By Lemma \ref{l:G-S_lem1}, $|S^{(j)}\cap W_i|\le 1$ for $i=1,2$ and $j=1,\dots,\ptp(G;S)$, so $|S^{(j)}|\le 2$ for  $j=1,\dots,\ptp(G;S)$.  Thus $\ptp(G;S)\ge \frac{n-k}2$, and $\thpa(G;S)=k\ptp(G;S)\ge k\frac {n-k}2\ge n-2$ for $k\ge 2$ by Lemma \ref{l:kp-alg}. If $k=1$, then $G$ is a tree,  $\alpha(G)=2$ implies $G=P_3$ or $G=P_4$, and thus $\thpa(G)=n-2$. 
  \epf


 \section{Comparisons of product throttling numbers}\label{s:compare} 

Recall that $\capt(G;S)\le \ptp(G;S)\le \pt(G;S)$ for any $S\subseteq V(G)$ (see Theorem \ref{t:ptc-le-ptp} for the first inequality), and thus \[\thcx(G)\le \thpx(G)\le \thzx(G)\mbox{ and }\thca(G)\le \thpa(G)\le\thza(G).\]
From the definitions, $\ppt(G;S)\le \pt(G;S)$ for any $S\subseteq V(G)$, so
\[\thpdx(G)\le \thzx(G)\mbox{ and } \thpda(G)\le\thza(G).\]
We establish the noncomparability of the following pairs of parameters:  $\thcx(G)$ and $\thpdx(G)$,   $\thca(G)$ and $\thpda(G)$,  $\thpx(G)$ and $\thpdx(G)$,  $\thpa(G)$ and $\thpda(G)$.

\begin{ex} Let $r,h\ge 2$. 
Recall that $T_{r,h}$ denotes a full $r$-ary tree  of height $h$ (see Proposition \ref{p:r-ary-tree}).
  Since $\Zp(T_{r,2})=c(T_{r,2})=1$ and $\capt(T_{r,2})=\ptp(T_{r,2})=2$, $\thpx(T_{r,2})=\thcx(T_{r,2})=3$ and $\thpa(T_{r,2})=\thca(T_{r,2})=2$.  Since $\pd(T_{r,2})=\gamma(T_{r,2})=r$, $\thpdx(T_{r,2})=2r>3$ and $\thpda(T_{r,2})=r>2$ for $r\ge 3$.
\end{ex}

\begin{ex}  The \emph{generalized wheel} $GW(k,r)$ with  $k\ge 4$ and $r\ge 1$ is the graph of order $kr+1$ constructed from $C_k\cp P_r$ by adding a new vertex  adjacent to every vertex in one end copy of $C_k$;  Figure \ref{f:gw62} shows $GW(6,2)$.  Let  $k\ge 5$, so $\gamma(GW(k,2))\ge 3$.  It is immediate that  $\thpdx(GW(k,2))=3$ and  $\thpda(GW(k,2))=2$, whereas $c(GW(k,2))=2$ and $\capt(GW(k,2))\ge 2$, so  $\thpx(GW(k,2))\ge \thcx(GW(k,2))\ge 6$ and $\thpa(GW(k,2))\ge \thca(GW(k,2))\ge 3$.
  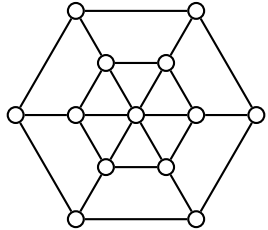
\begin{figure}[!h]
\begin{center}

 \begin{tikzpicture}[scale=.8]
\vertex (v0) at (0,0)
{};
\vertex (v1) at (.5,-.866)
 {};
\vertex (v2) at (1,0)
{};
\vertex (v3) at (.5,.866)
{};
\vertex (v4) at (-.5,.866)
{};
\vertex (v5) at (-1,0)
{};
\vertex (v6) at (-.5,-.866)
{};
\vertex (v7) at (1,-1.732)
{};
\vertex (v8) at (2,0) 
{};
\vertex (v9) at (1,1.732)
{};
\vertex (v10) at (-1,1.732)
{};
\vertex (v11) at (-2,0)
{};
\vertex (v12) at (-1,-1.732)
{};

\foreach \n in {1,2,3,4,5,6}{ \draw[thick] (v0) to (v\n); }
\foreach \n in {2,6,7}{ \draw[thick] (v1) to (v\n); }
\foreach \n in {2,4,9}{ \draw[thick] (v3) to (v\n); }
\foreach \n in {4,6,11}{ \draw[thick] (v5) to (v\n); }
\foreach \n in {2,7,9}{ \draw[thick] (v8) to (v\n); }
\foreach \n in {4,9,11}{ \draw[thick] (v10) to (v\n); }
\foreach \n in {6,7,11}{ \draw[thick] (v12) to (v\n); }
\end{tikzpicture}\caption{The generalized wheel $GW(6,2)$. \label{f:gw62}}
\vspace{-8pt}
\end{center}
\end{figure}
\end{ex}

From Theorem \ref{t:thza-not-throttle}, Theorem \ref{half}, and Observation \ref{o:basic-bds-up}, it follows that for any graph $G$, 
\[\thza(G)\geq {\frac n 2} \geq \gamma (G)\geq \thpda(G).\]

Moreover, there exist graphs $G$ such that $\thza(G)= {\frac n 2}= \gamma (G)= \thpda(G)$. For example, a straightforward verification shows $\thza(C_4)=2$, $\gamma (G)=2$ and  $\thpda(G)=2$. We now combine results from  previous sections to obtain the following characterization of graphs $G$ such that $\thza(G)= {\frac n 2}= \gamma (G)= \thpda(G)$.

\begin{prop}
Let $G$ be a  connected graph of order $n$ such that $\thpda(G)={\frac n 2}$. Then $\thza(G)= {\frac n 2}= \gamma (G)$.
\end{prop}

\bpf
By Theorem \ref{half-ratio-iff},  $\thpda(G)=\frac n 2$ if and only if $G=(H\circ K_1)\circ K_1$ for some connected graph $H$,  $G=C_4\circ K_1$, or $G=C_4$.  
To conclude the proof we  show  that $G$ is a matched-sum graph, and apply Theorem \ref{t:thza-throttle=halfeven} to conclude $\thza(G)= {\frac n 2}$. If $G=C_4$, then $G=H_1M^+H_2$ where $H_1=K_2$, $H_2=K_2$ and $M$ is any matching between $V(H_1)$ and $V(H_2)$.  Now suppose $G=H_1\circ K_1$ for some connected graph $H_1$. Let $V(H_2)=V(G)\setminus V(H_1)$. For each $v\in V(H_2)$,  $v$ has a unique neighbor $u_v$ in $G$; note that $u_v\in V(H_1)$. Let $M=\{vu_v:v\in V(H_2)\}$. Then $G$ is a matched-sum graph, specifically,  $G=H_1M^+H_2$. 
\epf

\begin{cor} 
A connected graph $G$ satisfies $\thza(G)= {\frac n 2}= \gamma (G)= \thpda(G)$ if and only if $G=(H\circ K_1)\circ K_1$ for some connected graph $H$,  $G=C_4\circ K_1$, or $G=C_4$.
\end{cor}


\section*{Acknowledgements}

The work of A. Trenk was partially supported by a grant from the Simons Foundation (\#426725).
The work of C. Mayer was partially supported by Sandia National Laboratories.  Sandia National Laboratories is a multimission laboratory managed and operated by National Technology \& Engineering Solutions of Sandia, LLC, a wholly owned subsidiary of Honeywell International Inc., for the U.S. Department of Energy's National Nuclear Security Administration under contract DE-NA0003525. This paper describes objective technical results and analysis. Any subjective views or opinions that might be expressed in the paper do not necessarily represent the views of the U.S. Department of Energy or the United States Government.


\bibliographystyle{amsalpha}

\end{document}